\documentclass[a4paper,11pt
]{amsart}
\usepackage{amsmath,amssymb,amsthm}
\usepackage[color = blue!20, bordercolor = black, textsize = tiny]{todonotes}
\usepackage{relsize, comment} 
\usepackage[bbgreekl]{mathbbol} 

 \usepackage[ left=3cm, right=3cm, top = 2.8cm, bottom = 2.8cm ]{geometry}
\makeatletter
\def\@tocline#1#2#3#4#5#6#7{\relax
  \ifnum #1>\c@tocdepth 
  \else
    \par \addpenalty\@secpenalty\addvspace{#2}%
    \begingroup \hyphenpenalty\@M
    \@ifempty{#4}{%
      \@tempdima\csname r@tocindent\number#1\endcsname\relax
    }{%
      \@tempdima#4\relax
    }%
    \parindent\z@ \leftskip#3\relax \advance\leftskip\@tempdima\relax
    \rightskip\@pnumwidth plus4em \parfillskip-\@pnumwidth
    #5\leavevmode\hskip-\@tempdima
      \ifcase #1
      \or\or \hskip 2em \or \hskip 2em \else \hskip 3em \fi%
      #6\nobreak\relax
    \dotfill\hbox to\@pnumwidth{\@tocpagenum{#7}}\par
    \nobreak
    \endgroup
  \fi}
\makeatother

\usepackage[colorlinks,linkcolor=blue,citecolor=blue,urlcolor=black]{hyperref}
\usepackage{tikz-cd}

\numberwithin{equation}{section}
\newtheorem{thm}[subsection]{Theorem}

\newtheorem{lem}[subsection]{Lemma}
\newtheorem{prop}[subsection]{Proposition}
\theoremstyle{definition}
\newtheorem{df}[subsection]{Definition}
\newtheorem{rmk}[subsection]{Remark}

\newtheorem{const}[subsection]{Construction}
\newtheorem{quest}[subsection]{Question}

\newcommand{\bE}{\mathbf{E}}

\newcommand{\bM}{\mathbf{M}}

\newcommand{\bW}{\mathbf{W}}

\newcommand{\A}{\mathbb{A}}

\newcommand{\E}{\mathbb{E}}

\newcommand{\G}{\mathbb{G}}

\renewcommand{\L}{\mathbb{L}}

\renewcommand{\P}{\mathbb{P}}
\newcommand{\Q}{\mathbb{Q}}

\newcommand{\Sph}{\mathbb{S}}

\newcommand{\Z}{\mathbb{Z}}

\newcommand{\cC}{\mathcal{C}}
\newcommand{\cD}{\mathcal{D}}

\newcommand{\cF}{\mathcal{F}}
\newcommand{\cG}{\mathcal{G}}
\newcommand{\cH}{\mathcal{H}}

\newcommand{\cM}{\mathcal{M}}

\newcommand{\cO}{\mathcal{O}}

\newcommand{\rB}{\mathrm{B}}

\newcommand{\rD}{\mathrm{D}}

\newcommand{\rL}{\mathrm{L}}

\newcommand{\rN}{\mathrm{N}}

\newcommand{\kX}{\mathfrak{X}}

\newcommand{\et}{\mathrm{\acute{e}t}}

\DeclareMathOperator{\Hom}{Hom}
\DeclareMathOperator{\Spec}{Spec}
\DeclareMathOperator{\Spf}{Spf}

\DeclareMathOperator{\chara}{char}
\newcommand{\colim}{\mathop\mathrm{colim}}
\newcommand{\id}{\mathrm{id}}
\newcommand{\ul}{\underline}

\newcommand{\lSm}{\mathrm{lSm}}
\newcommand{\lSch}{\mathrm{lSch}}
\newcommand{\Sm}{\mathrm{Sm}}
\newcommand{\SmlSm}{\mathrm{SmlSm}}

\newcommand{\Fil}{\mathrm{Fil}}
\newcommand{\fil}{\mathrm{f}}
\newcommand{\gr}{\mathrm{gr}}
\DeclareMathOperator{\Fun}{Fun}
\DeclareMathOperator{\fib}{fib}
\DeclareMathOperator{\cofib}{cofib}
\DeclareMathOperator{\im}{im}
\newcommand{\DF}{\mathrm{DF}}
\newcommand{\lQSyn}{\mathrm{lQSyn}}
\newcommand{\FlQSyn}{\mathrm{FlQSyn}}

\newcommand{\lQRSPerfd}{\mathrm{lQRSPerfd}}
\newcommand{\Poly}{\mathrm{Poly}}

\newcommand{\Sh}{\mathrm{Sh}}
\newcommand{\Sp}{\mathrm{Sp}}
\newcommand{\Mod}{\mathrm{Mod}}
\DeclareMathOperator{\map}{map}

\newcommand{\eff}{\mathrm{eff}}
\newcommand{\veff}{\mathrm{veff}}

\newcommand{\logSH}{\mathrm{logSH}}
\newcommand{\logDA}{\mathrm{logDA}}
\newcommand{\SH}{\mathrm{SH}}

\newcommand{\Kth}{\mathrm{K}}
\newcommand{\HH}{\mathrm{HH}}
\newcommand{\HC}{\mathrm{HC}}
\newcommand{\HP}{\mathrm{HP}}
\newcommand{\THH}{\mathrm{THH}}
\newcommand{\TC}{\mathrm{TC}}
\newcommand{\TP}{\mathrm{TP}}
\newcommand{\BMS}{\mathrm{BMS}}
\newcommand{\HKR}{\mathrm{HKR}}
\newcommand{\Tr}{\mathrm{Tr}}
\newcommand{\Pic}{\mathrm{Pic}}
\newcommand{\Hod}{\mathrm{Hod}}
\newcommand{\dR}{\mathrm{dR}}
\newcommand{\Tot}{\mathrm{Tot}}

\newcommand{\crys}{\mathrm{crys}}
\newcommand{\Zar}{\mathrm{Zar}}
\newcommand{\seta}{\mathrm{s\acute{e}t}}
\newcommand{\sproet}{\mathrm{spro\acute{e}t}}

\newcommand{\sNis}{\mathrm{sNis}}
\newcommand{\syn}{\mathrm{syn}}

\newcommand{\letale}{\mathrm{l\acute{e}t}}
\newcommand{\ket}{\mathrm{k\acute{e}t}}
\newcommand{\setale}{\mathrm{s\acute{e}t}}
\newcommand{\mot}{\mathrm{mot}}

\newcommand{\bHH}{\mathbf{HH}}
\newcommand{\bHC}{\mathbf{HC}}
\newcommand{\bHP}{\mathbf{HP}}
\newcommand{\bTHH}{\mathbf{THH}}
\newcommand{\bTC}{\mathbf{TC}}
\newcommand{\bTP}{\mathbf{TP}}
\newcommand{\slice}{\mathrm{s}}
\newcommand{\MZ}{\mathbf{M}\mathbb{Z}}
\newcommand{\KGL}{\mathbf{KGL}}
\newcommand{\bOmega}{\mathbf{\Omega}}
\newcommand{\MS}{\mathrm{MS}}

\DeclareSymbolFontAlphabet{\mathbb}{AMSb} 
\DeclareSymbolFontAlphabet{\mathbbl}{bbold} 
\newcommand{\cPrism}{\widehat{\mathlarger{\mathbbl{\Delta}}}}
\newcommand{\bPrism}{\widehat{\mathbf{\Delta}}}
\newcommand{\can}{\mathrm{can}}

\newcommand{\conj}{\mathrm{conj}}

\title{On the logarithmic slice filtration}

\author{Federico Binda}
\address{Department of Mathematics ``F. Enriques'', University of Milan, Italy}
\email[F. Binda]{federico.binda@unimi.it}

\author{Doosung Park}
\address{Department of Mathematics and Informatics, University of Wuppertal, Germany}
\email{dpark@uni-wuppertal.de}

\author{Paul Arne {\O}stv{\ae}r}
\address{Department of Mathematics ``F. Enriques'', University of Milan, Italy \&
Department of Mathematics, University of Oslo, Norway}
\email{paul.oestvaer@unimi.it \& paularne@math.uio.no}

\date{\today}

\begin{document}

\begin{abstract}
We consider slice filtrations in logarithmic motivic homotopy theory. 
Our main results establish conjectured compatibilities with the Beilinson, BMS, 
and HKR filtrations on (topological, log) Hochschild homology and related invariants.
In the case of perfect fields admitting resolution of singularities, 
we show that the slice filtration realizes the BMS filtration on the $p$-completed topological cyclic homology. 
Furthermore,
the motivic trace map is compatible with the slice and BMS filtrations, 
yielding a natural morphism from the motivic slice spectral sequence to the BMS spectral sequence. 
Finally, 
we consider the Kummer \'etale hypersheafification of logarithmic $K$-theory and show that its very 
effective slices compute Lichtenbaum \'etale motivic cohomology.
\end{abstract}

\maketitle

\section{Introduction}

Logarithmic motives provide a framework for studying non-$\A^1$-invariant cohomology theories 
in arithmetic geometry 
\cite{BLMP}, 
\cite{BLPO}, 
\cite{BLPO2}, 
\cite{Binda_Merici}, 
\cite{logDMcras}, \cite{BPO}, \cite{BPO2}, 
\cite{Merici}, 
\cite{SaitoRSClog}.
This paper introduces logarithmic analogs of slice filtrations in motivic homotopy theory 
defined by Voevodsky \cite{MR1977582} and Spitzweck-{\O}stv{\ae}r \cite{SO:twistedK}. 
The slice perspective produces explicit calculations, 
see \cite{zbMATH07303324} for a survey, 
and remains a powerful tool in homotopy theory following Hill-Hopkins-Ravenel's 
solution of the Kervaire invariant one problem \cite{zbMATH06605831}.
Examples of applications of motivic slices include 
a proof of the Milnor conjecture on quadratic forms in \cite{zbMATH06578603} 
and calculations of universal motivic invariants in \cite{zbMATH07003144}, 
\cite{röndigs2021second}.
One of the main purposes of this paper is to compare
the logarithmic slice filtration with the Bhatt-Morrow-Scholze filtration on 
topological Hochschild homology and refinements \cite{BMS19}.

Assuming $k$ is a perfect field of characteristic $p$ and $A$ is a smooth $k$-algebra, 
Bhatt, Morrow, and Scholze \cite[\S 1.4]{BMS19} expected that their filtrations on topological Hochschild homology 
$\THH(A;\Z_p)$ (and the related theories $\TC^-(A;\Z_p)$, $\TP(A;\Z_p)$, and $\TC(A;\Z_p)$) 
afford a precise relation with filtrations on algebraic $K$-theory $\Kth(A)$, 
see \cite{FS}, \cite{MR2365658}, 
via the trace maps
\[
\Kth(A) \to \THH(A;\Z_p),\; \TC^-(A;\Z_p),\; \TP(A;\Z_p),\; \TC(A;\Z_p).
\]
In the literature, these filtrations are often called ``motivic filtrations''. 
Mathew verified this expectation in \cite[p.\ 4]{zbMATH07729746} by showing that the filtrations are compatible using Postnikov towers in (pro-)Nisnevich and (pro-)\'etale topologies.
Furthermore, Mathew raised a deeper question regarding the existence of a unified construction that could realize both filtrations.

This paper aims to explore the motivic properties of these filtrations using logarithmic motivic homotopy theory developed in \cite{BPO2} and provide a first positive answer in this direction. 
A brief recapitulation of Voevodsky's slice filtration in $\A^1$-homotopy theory will help 
clarify our approach.
In algebraic topology, a standard method to understand a spectrum $E$  is to study its Postnikov tower, i.e., \ how it is built out of topological Eilenberg-MacLane spectra. Sending $E$ to its $n-1$-connected cover $E^{(n)}$ defines a functor from $\SH$ to the full subcategory $\Sigma^n \SH_{\geq 0}$ of $n-1$-connected spectra, that is right adjoint to the obvious inclusion. 
Replacing the category $\Sigma^n \SH_{\geq 0}$ with the category $\Sigma^{2n,n}\SH(S)^{\rm eff}$ measuring 
$\G_m$-effectivity, Voevodsky defined a motivic analog of the Postnikov tower, known as the slice tower 
\[
{\fil}_{n+1} E  \to {\fil}_{n} E \to {\fil}_{ n-1 } E \to \cdots \to E
\]
for any motivic spectrum $E$. When $E= \KGL$ is the spectrum representing algebraic $K$-theory, 
the induced spectral sequence is the motivic Atiyah-Hirzebruch spectral sequence of Levine-Voevodsky 
\cite{MR2365658},
\cite{zbMATH01803815}.

Since (topological) Hochschild homology and its variants are not $\A^1$-homotopy invariant (in fact, the $\A^1$-localization of $\THH$ is trivial, see, e.g., \cite{EEnonA1}), there is no motivic spectrum in $\SH(S)$ representing them and thus Voevodsky's slice machinery is inapplicable. In \cite{BPO2}, we developed an extension of $\SH(S)$ using the language of log schemes and proved that there are indeed $\P^1$-spectra representing $\THH$ and variants. The HKR filtration, the Beilinson filtration, and the Bhatt-Morrow-Scholze filtration on $\HH$, $\HC^{-}$, $\THH$, $\TC^-$ and so on are all compatible with the bounding maps defining the motivic spectra $\bHH$, $\bHC^{-}$, $\bTHH$, $\bTC^{-}$ and relatives, giving rise to filtered $\P^1$-motivic spectra. See Propositions \ref{HKR.6} and \ref{BMS.2}. In \cite[\S 6.5.6]{BPO2} we have also constructed a $\P^1$-motivic spectrum $\KGL$ (denoted $\mathbf{logKGL}$ in \cite{BPO2}), representing algebraic $K$-theory, equipped with a motivic lift of the (logarithmic) cyclotomic trace map when $S=k$ is a perfect field admitting resolution of singularities. See \cite[Corollary 8.6.2]{BPO2}. 

In this work, we propose an analog of Voevodsky's slice filtration within the context of $\logSH(S)$ and its linearized variant $\logDA(S)$.
We define the effective category $\logSH(S)^{\rm eff}$ as the localizing subcategory of $\logSH(S)$ generated by $\P^1$-suspensions of $(X, \partial X)\in \SmlSm/S$, where $X\in \Sm/S$ and $\partial X$ is a strict normal crossing divisor on $X$. Similarly, we define the coeffective category $\logSH(S)^{\rm coeff}$ as its right orthogonal in the obvious sense. 
Through (de)suspension, we obtain the $n$-effective and the $m$-coeffective categories for every $n,m\in \Z$. Following Spitzweck-{\O}stv{\ae}r \cite{SO:twistedK}, we will consider the very effective version of the slice filtration in $\logSH(S)$. Rather than measuring $\G_m$-effectivity, the very effective slice filtration measures $\P^1$-effectivity and has strong convergence properties. Our main results are the following.
Below,
when $\Fil_\bullet \cF$ is a filtration on $\cF$,
we set $\Fil^n \cF:=\cofib(\Fil_{n+1}\cF\to \cF)$ for $n\in \Z$.

\begin{thm}[See Theorem \ref{HKR.3}]\label{thm:introHC}Let $R$ be any ring, and let $S\to \Spec(R)$ be a smooth $R$-scheme. 
For $n\in \Z$, the complexes
\[
\Fil_{\HKR}^n\bHH(-/R),
\;
\Fil_{\rB}^n\bHC^-(-/R),
\;
\Fil_{\rB}^n\bHP(-/R)
\]
are $n$-co-very effective, that is, they belong to the subcategory $\logDA(S)_{\leq n}^{\veff}$ of Construction \ref{constr.veff}. In particular, there are canonical morphisms 
\begin{align*}
&\tilde{\fil}_n \bHH(-/R) \simeq \Sigma^{2n,n}\tilde{\fil}_0  \bHH(-/R)  \to \Fil_n^{\HKR} \bHH(-/R),
\end{align*}
and similarly for $\bHC^-$ and $\bHP$, where $\tilde{\fil}_n$ denotes the $n$th very effective slice functor in $\logDA(S)$. 
\end{thm}

\begin{thm}[See Theorem \ref{BMS.3}]\label{thm:introTC} Let $S$ be the spectrum of a $p$-adic quasi-syntomic ring $R$. 
For $n\in \Z$, the filtered spectra
\[
\Fil_{\BMS}^n\bTHH(-;\Z_p),
\;
\Fil_{\BMS}^n\bTC^-(-;\Z_p),
\;
\Fil_{\BMS}^n\bTP(-;\Z_p),
\;
\Fil_{\BMS}^n\bTC(-;\Z_p)
\]
are in $\logSH(S)_{\leq n}^{\veff}$. In particular, there are canonical morphisms
\[
\tilde{\fil}_n \bTHH(-;\Z_p) \simeq \Sigma^{2n,n} \tilde{\fil}_0 \bTHH(-;\Z_p) \to \Fil_n^{\BMS} \bTHH(-;\Z_p), 
\]
and similarly for $\bTC^{-}$, $\bTP$ and $\bTC$. 
\end{thm}
Assume now that $S=\Spec(k)$ is a field of positive characteristic $p$, admitting resolution of singularities. Since topological cyclic homology and variants are, in any case, étale sheaves, we can consider an étale variant of the very effective slice filtration in the subcategory spanned by \emph{$p$-completed objects} in $\logSH(k)$. Our result is then the following. 

\begin{thm}[See Theorem \ref{ket.5}]
Let $k$ be a perfect field admitting resolution of singularities.
Then, the induced morphism
\[
\tilde{\fil}_i^\ket\bTC(-;\Z_p)
\to
\Fil_i^\BMS\bTC(-;\Z_p)
\]
is an equivalence in the $\infty$-category of $p$-complete Kummer \'etale motives $\logSH_\ket^\wedge(k)_p^\wedge$ for every integer $i$. Here, we denote by $\tilde{\fil}_i^\ket$ the $i$th very effective cover in $\logSH_\ket^\wedge(k)_p^\wedge$.
\end{thm}

In other words, despite having very different origins, we can identify the slice and the BMS filtration. We see this as an answer to Mathew's question for $\TC(-;\Z_p)$,
i.e.,
the slice filtrations realize the two ``motivic filtrations'' on $\Kth$ and $\TC(-;\Z_p)$. We remark that the proof of the above theorem crucially relies on Geisser-Levine \cite{MR1738056}, together with a motivic version of the prismatic-crystalline comparison \cite{BLMP}.

It is a natural question to compare the slice towers in $\logSH(S)$ and $\SH(S)$. 
When the base $S$ is the spectrum of a field that admits resolution of singularities, we can combine the above results with the trace methods and obtain the following.

\begin{thm}
\label{intro.1}
Let $k$ be a perfect field of characteristic $p>0$ admitting resolution of singularities.
\begin{enumerate}
\item The log motivic trace maps in $\logSH(k)$
\begin{gather*}
\KGL \to \bTHH(-;\Z_p),\quad 
\KGL \to \bTC^-(-;\Z_p),
\\
\KGL \to \bTP(-;\Z_p),\quad 
\KGL \to \bTC(-;\Z_p)
\end{gather*}
are compatible with the slice filtrations on the left-hand sides and the BMS filtrations on the right-hand sides.
\item
The natural maps
\[
\KGL \to
\bTC(-;\Z_p)
\to
\bTC^-(-;\Z_p)
\to
\bTP(-;\Z_p),\bTHH(-;\Z_p)
\]
on the graded pieces give rise to the natural maps
\[
\MZ(i)[2i]\to L_{\sproet}\bW\bOmega^i_{\log}[i]\to \Fil^\rN_i\bW\bOmega[2i]\to \bW\bOmega[2i], \Fil^\conj_i\bOmega[2i].
\]
\item
The graded pieces of the log trace maps yield natural morphisms for every $X\in \SmlSm/k$
\begin{gather*}
R\Gamma_{\mot}(X-\partial X,\Z(i))
\to
R\Gamma_{\seta}(X,\tau^{\leq i}\Omega),
\\
R\Gamma_{\mot}(X-\partial X,\Z(i))
\to
R\Gamma_{\seta}(X,\Fil_i^\rN W\Omega),
\\
R\Gamma_{\mot}(X-\partial X,\Z(i))
\to
R\Gamma_{\seta}(X,W\Omega),
\\
R\Gamma_{\mot}(X-\partial X,\Z(i))
\to
R\Gamma_{\sproet}(X,W\Omega^i_{\log}[-i]),
\end{gather*}
where the left-hand side denotes Voevodsky-Suslin-Friedlander motivic cohomology with integral coefficients \cite{VSF}. 
\end{enumerate}
\end{thm}

In light of Levine's comparison  between the homotopy coniveau tower and the slice tower \cite[Theorem 9.0.3]{MR2365658}, we obtain in particular that for $X\in \SmlSm/k$,
the log trace maps
\begin{gather*}
\Kth(X-\partial X) \to \THH(X;\Z_p), \quad 
\Kth(X-\partial X) \to \TC^-(X;\Z_p),
\\
\Kth(X-\partial X) \to \TP(X;\Z_p), \quad 
\Kth(X-\partial X) \to \TC(X;\Z_p)
\end{gather*}
are compatible with the homotopy coniveau filtrations on the left-hand sides and the BMS filtrations on the right-hand sides.
As an immediate corollary,
we obtain a natural morphism from the motivic Atiyah-Hirzebruch spectral sequence \cite[\S 11.3]{MR2365658}
\[
E_2^{pq}=H_{\mot}^{i-j}(X-\partial X,\Z(-j))
\Rightarrow
\pi_{-i-j}\Kth(X-\partial X)
\]
to the log version of the BMS spectral sequence \cite[Theorem 1.12(5)]{BMS19}, (see \cite[Theorem 1.3]{BLPO2} for the log case)
\[
E_2^{ij}=H^{i-j}_{\syn}(X,\Z_p(-j))
\Rightarrow
\pi_{-i-j}\TC(X;\Z_p)
\]
for $X\in \SmlSm/k$ and in particular for $X\in \Sm/k$.
In particular, we get directly that the morphism
\[H_{\mot}^{i}(X-\partial X,\Z(j)) \to   H^{i}_{\syn}(X-\partial X,\Z_p(j))\]
considered in \cite[Corollary 8.21]{BMS19}, \cite[Theorem 6.15]{BeilFiber} (using Geisser-Levine to identify $p$-adic motivic cohomology) factors canonically through the \emph{log} syntomic cohomology $H^{i}_{\syn}(X,\Z_p(j))$. Note that the input of \cite{MR1738056} is not necessary to obtain the map: it is just a consequence of the compatibility between the slice filtration and the BMS filtration, together with \cite[Theorem 6.4.2]{MR2365658}. For $p$-adic motivic cohomology with $\Q_p$-coefficients, a similar refinement has been considered by Ertl and Niziol in \cite{zbMATH07020394} (see Remark \ref{rmk:Licht_mot}). The comparison between motivic cohomology and syntomic cohomology (the non-log case) is also considered in the recent work of Elmanto-Morrow \cite{Elmanto_Morrow}.

The non-topological counterpart (for classical cyclic, periodic, and Hochschild homology) also holds.

\begin{thm}
\label{intro.2}
Let $k$ be a perfect field of admitting resolution of singularities (e.g., $\chara k=0$).
\begin{enumerate}
\item The log motivic trace maps in $\logSH(k)$
\begin{gather*}
\KGL \to \bHH(-/k),\quad 
\KGL \to \bHC^-(-/k),\quad 
\KGL \to \bHP(-/k)
\end{gather*}
are compatible with the slice filtrations on the left-hand sides, the HKR filtration on $\bHH$,  and the Beilinson filtrations on $\bHC^-$ and $\bHP$.
\item The graded pieces of the log motivic trace maps are identified with
\begin{gather*}
\MZ(i)[2i]
\to
\bOmega^i[i],
\\
\MZ(i)[2i]
\to
\bOmega^{\geq i}[2i],
\\
\MZ(i)[2i]
\to
\bOmega[2i].
\end{gather*}
\item The graded pieces of the log trace maps yield natural morphisms
\begin{gather*}
R\Gamma_{\mot}(X-\partial X,\Z(i))
\to
R\Gamma_{\sNis}(X,\bOmega^i)[-i],
\\
R\Gamma_{\mot}(X-\partial X,\Z(i))
\to
R\Gamma_{\sNis}(X,\bOmega^{\geq i}),
\\
R\Gamma_{\mot}(X-\partial X,\Z(i))
\to
R\Gamma_{\sNis}(X,\bOmega).
\end{gather*}
\end{enumerate}
\end{thm}
We remark that the compatibility between the slice filtration and the BMS filtration along the trace map has been established using completely different methods (see again \cite{Elmanto_Morrow} and the upcoming work of Bachmann-Elmanto-Morrow). 
Our results refine such statements (at least over a field), since the trace map further factors through the log invariants, in a way that is compatible with the motivic filtrations on the source and 
target.

Finally, in the last section of the paper, we consider the slice filtration in the (hypercomplete) Kummer \'etale version $\logSH^{\wedge}_{\ket}(k)$ of $\logSH(k)$. 

\begin{thm}[See Theorem \ref{thm:KummerLich}]
Let $k$ be a perfect field admitting resolution of singularities,
and assume that the \'etale cohomological dimension of $k$ is finite.

There are natural equivalences
\[
\slice_i^\ket L_{\ket}\KGL
\simeq
\tilde{\slice}_i^\ket L_{\ket}\KGL
\simeq
L_{\ket}\slice_i\KGL
\simeq
L_{\ket}\tilde{\slice}_i\KGL
\simeq
\Sigma^{2i,i}L_{\ket}\MZ
\]
in $\logSH_{\ket}^\wedge (k)$. Moreover, $L_{\ket} \MZ$ represents Lichtenbaum \'etale motivic cohomology \cite[Lecture 10]{MVW} when restricted to $\Sm/k$. 
\end{thm}
For $X\in \SmlSm/k$ note moreover that
there is a natural equivalence
\[
\map_{\logSH_{\ket}^\wedge(k)}(\Sigma^\infty X_+,L_{\ket}\KGL)
\simeq
L_{\ket}\Kth(X),
\]
where the right hand-side is the Kummer \'etale hypersheafification of the functor $X\mapsto \Kth(X-\partial X)$ for $X\in \SmlSm/k$. In particular, for $X\in \Sm/k$ is just \'etale $K$-theory.
\begin{rmk}
Theorem \ref{intro.1} provides in particular a natural map for every $X\in \SmlSm/k$ 
\[ R\Gamma(X-\partial X, \Z(i)) \to R\Gamma_{\sproet}(X,W\Omega^i_{\log}[-i]).\]
    As pointed out by a referee, one could also consider the composite morphism
    \[ 
    R\Gamma(X-\partial X, \Z(i)) \to  R\Gamma(X-\partial X, \tau_{\geq i} \Z(i)) \to  R\Gamma(X-\partial X, \mathcal{K}^M_i[-i]) \xrightarrow{\mathrm{dlog}} R\Gamma_{\sproet}(X,W\Omega^i_{\log}[-i]),
    \]
    where $\mathcal{K}^M_i$ is the sheaf of Milnor $K$-theory, and the dlog map can be constructed using e.g.~the argument in \cite[\S 1]{JSZ}. It is reasonable to expect that the maps agree. However, this does not follow directly from our results since one would have to explicitly check the shapes of the cyclotomic trace and the resulting map after applying the slice filtration. We leave this problem to the interested reader. 
\end{rmk}

\begin{rmk}[On resolution of singularities] For the $K$-theoretic applications we have assumed to be working over a field $k$ admitting resolution of singularities. This is used in the explicit computation of the right adjoint $\omega^*$ from Voevodsky's motives to log motives, that on a smooth $k$-scheme $X$ satisfies $\omega^*(X) = (\overline{X}, \partial X)$ where $\overline{X}$ is a smooth proper compactification of $X$, with normal crossing boundary $\partial X$. 
This simple formula allows us, 
for example, 
to identify $\omega^* \mathbf{KGL}$ as well as $\omega^* \MZ$, and to show that $\omega^*$ commutes with taking (very effective) slices. Using the above, we can directly invoke the known results in $\SH(k)$.
    One may ask about compatibility between $\omega^*$ and $\fil_n$ without assuming resolution of singularities. We expect that the analog of Levine's computation holds for the log $K$-theory spectrum $\mathbf{logKGL}$ defined in \cite[\S 6.5]{BPO2}. We plan to address this question in future work. 
\end{rmk}

\subsection{Outline of the proofs}Let us discuss the main idea of the proof of Theorems \ref{thm:introHC} and \ref{thm:introTC}. In \cite{BPO2}, we constructed motivic spectra in $\logSH(S)$ representing logarithmic Hochschild, periodic, and cyclic homology. Such spectra are Bott periodic spectra, in the sense that the strict Nisnevich sheaf $\HH(-)$ of spectra on $\lSch/S$ satisfies $\Omega_{\P^1} \HH \simeq \HH$ (and similarly for $\HP(-)$ and  $\HC^{-}(-)$). The compatibility of the $\P^1$-bundle formula with the HKR and the Beilinson filtration allows us to show (see Proposition \ref{HKR.4}) that we can promote the construction and obtain filtered motivic $\P^1$-spectra, satisfying the equivalence
\[\Fil_i^{\HKR}\HH(-/R)(1)[2] \simeq \Fil^{\HKR}_{i+1}\HH(-/R)
\]
and similarly for $\HP$ and $\HC^{-}$ with the Beilinson filtration, for any  ring $R$. At this point, the co-effectivity of $\Fil_{\HKR}^n \HH$ with respect to the $t$-structure given by the very effective slice tower is a direct consequence of the vanishing of the sheaf cohomology $H^{2j}(X, \Omega^j_{X/R})$ for negative $j$. 

The topological counterpart, that is, for $\THH$ and variant, is proven by a conceptually similar method. More precisely, one shows using the $\P^1$ bundle formula that the BMS filtration on $\THH$ gives an equivalence 
\[\Fil_i^{\BMS} \THH(-;\Z_p)(1)[2] \simeq  \Fil^{\BMS}_{i+1} \THH(-;\Z_p),\]
and similarly for $\TP$, $\TC^{-}$ and $\TC$, see Proposition \ref{BMS.5}. The co-effectivity statement can then be deduced by passing to the associated graded, reducing to vanishing in (log) prismatic cohomology. This is shown by log quasi-syntomic descent. Overall, the co-effectivity results immediately imply the existence of the natural morphisms between the very effective slice tower and the BMS, HKR, and Beilinson filtrations. 

Let us now consider the natural functor $\omega^*\colon \SH(S) \to \logSH(S)$, right adjoint to the $\A^1$-localization functor from $\logSH(S)$ to $\SH(S)$ for $S$ a scheme. 
When $S=\Spec(k)$ is a perfect field that admits resolution of singularities, it is easy to show that $\omega^*$ commutes with the (very effective) slice filtration on both sides. Using the results in \cite[\S 7.7]{BPO2} and the main result of Levine in \cite[Theorems 6.4.2, 9.0.2]{MR2365658}, 
we can then identify the very effective slices of $\omega^* \KGL$ with motivic cohomology in $\logSH(k)$. Combining this with Theorems \ref{thm:introHC} and \ref{thm:introTC} we obtain Theorems \ref{intro.1} and \ref{intro.2}. 

Finally, in \S \ref{sec:Kummer}, we consider the very effective slices of Kummer \'etale $K$-theory. This allows us to consider a spectrum representing Lichtenbaum \'etale motivic cohomology in $\logSH_{\ket}(k)$: if $k$ admits resolution of singularities, this is computed by the zeroth very effective slice of the Kummer \'etale sheafification of $\KGL$. See Theorem \ref{thm:KummerLich}. Using Geisser-Levine as input, together with the identification $\MZ_p^{\syn}(i) \simeq L_{\sproet}\bW\bOmega_{\log}^i[-i]$ as motivic spectra from \cite{BLMP}, we can further show that in the Kummer étale (or in the strict étale) $p$-complete category, we have $\Fil_i^\BMS \bTC(-;\Z_p)\in (\logSH_{\ket}^\wedge (k)_p^\wedge)_{\geq i}^\veff$. Since we have already proved in general that $\Fil_i^\BMS \bTC(-;\Z_p)$ is also $i$-co-very effective, we conclude.

\subsection*{Acknowledgements} We thank Marc Levine and Alberto Merici for inspiring discussions and Matthew Morrow for helpful comments on a preliminary version of the paper.  We are also grateful to the referee for their report, which significantly helped improve the exposition. F.B.\ is partially supported by the PRIN 2022 ``The arithmetic of motives and L-functions'' at MUR (Italy).  
D.P.\ is partially supported by the research training group GRK 2240 ``Algebro-Geometric Methods in Algebra, 
Arithmetic and Topology.'' P.A.\O.\ was partially supported by 
The European Commission -- Horizon-MSCA-PF-2022 
``Motivic integral $p$-adic cohomologies", 
the RCN Project no.\ 312472 ``Equations in Motivic Homotopy Theory", 
and a Guest Professorship awarded by The Radboud Excellence Initiative.

\section{Slices for logarithmic motives}

Throughout this section,
let $S$ be a quasi-compact and quasi-separated scheme,
and let $\tau$ be any topology on the category $\lSch/S$ generated by a family of morphisms in $\lSm/S$.
Here, 
$\lSch/S$ denotes the category of fs log schemes of finite type over $S$, 
and $\lSm/S$ denotes the category of log smooth fs log schemes of finite type over $S$.
Recall that the category $\SmlSm/S$ is the full subcategory of $\lSm/S$ consisting of $X\in \lSm/S$ such that $\underline{X}\in \Sm/S$. By \cite[Remark 2.3.19]{BPO2}, \cite[Lemma A.5.10]{BPO}, it can be equivalently described as the category of pairs $X=(\ul{X},\partial X)$ such that $\ul{X}\in \Sm/S$, $X-\partial X$ is an open subscheme of $\ul{X}$, and the closed complement $\partial X$ is a strict normal crossing divisor on $\ul{X}$. We see $(\ul{X},\partial X)$ as a log scheme with the compactifying log structure given by the embedding $X-\partial X \subset \underline{X}$.

We work with the Nisnevich sheaf of stable presentable symmetric monoidal $\infty$-categories
\[
\mathrm{Sch}^\mathrm{op} \to  \mathrm{CAlg}(\mathrm{Pr}_{\Mod_\Lambda}^{\mathrm{L}}), \quad S\mapsto \logSH_\tau(S, \Lambda)^\otimes
\]
given by
\[
\logSH_\tau(S,\Lambda)^\otimes
:=
\Sp_{\P^1}(\Sh_\tau(\SmlSm/S,\Mod_\Lambda)[(\P^\bullet,\P^{\bullet-1})^{-1}])^\otimes
\]
introduced in \cite[\S 3.5]{BPO2},
where $\Lambda$ is an $\E_\infty$-ring,
and $\Mod_\Lambda$ denotes the $\infty$-category of $\Lambda$-module objects in $\Sp$. Note that this $\infty$-category is equivalent to the model obtained by looking at $(\P^1,\infty)$-local dividing Nisnevich sheaves on $\lSm/S$ by \cite[Corollary 3.5.8]{BPO2}. 
It admits a symmetric monoidal structure and is presentable and stable by \cite[Theorems 4.5.2.1(1),  7.1.2.1, Corollary 7.1.1.5]{HA}.
To simplify the notation, we will often omit the superscript $\otimes$.
If $\Lambda$ is the sphere spectrum $\Sph$,
then we omit $\Lambda$ in $\logSH_\tau(S,\Lambda)$.
If $\Lambda$ is the Eilenberg-MacLane spectrum $HR$ for a commutative ring $R$,
then we set
\[
\logDA_\tau(S)
:=
\logSH_\tau(S,H\Z)
\text{ and }
\logDA_\tau(S,R)
:=
\logSH_\tau(S,H R).
\]
If $\tau$ is the strict Nisnevich topology,
then we omit the subscript $\tau$.

We will be using the following conventions for filtrations.
Let $\cF$ be an object of an $\infty$-category $\cC$.
The \emph{$\infty$-category of filtered objects of $\cC$} is $\Fun(\Z^{op},\cC)$,
where $\Z^{op}$ is the category consisting of the integers such that $\Hom_{\Z^{op}}(i,j)$ is $*$ if $i\geq j$ and $\emptyset$ if $i<j$.
A filtered object $\cF(-)$ of $\cC$ is \emph{complete} if $\lim_i \cF(i)\simeq 0$.
The \emph{underlying object of $\cF(-)$} is $\cF(-\infty):=\colim_i \cF(i)$.
The filtration $\cF(-)$ on $\cF(-\infty)$ is \emph{exhaustive}.
We write 
\begin{gather*}
\Fil_i\cF:=\cF(i),
\\
\Fil^i\cF:=\cofib(\Fil_{i+1}\cF\to \cF),
\\
\gr_i \cF:=\cofib(\Fil_{i+1}\cF\to \Fil_i\cF).
\end{gather*}
Observe that we have $\gr_i\cF\simeq \fib(\Fil^i \cF\to \Fil^{i-1}\cF)$.

The slice filtration on $\SH(S)$ is due to Voevodsky \cite[\S 2]{MR1977582}.
We have its direct analog on $\logSH(S)$ as follows.
\begin{const}
\label{slice.1}
Let $\logSH_\tau(S,\Lambda)^\eff \subset \logSH_\tau(S,\Lambda)$ be the category of \emph{effective spectra},
that is, the full stable $\infty$-subcategory of $\logSH_\tau(S, \Lambda)$ generated under
colimits by $\Sigma^{n,0}\Sigma^\infty_{\P^1} X_+$; 
here $\Sigma^\infty_{\P^1} X_+$ (denoted $\Sigma^{\infty}_T X_+$ in \cite{BPO2}) is the $\P^1$-suspension spectrum of $X\in \SmlSm/S$ equipped with an extra 
base-point, $n$ is an integer, 
and $\Sigma^{p,q}$ tensoring with the log motivic sphere 
$S^{p-q}\otimes (\A^1/(\A^1,0))^{\otimes q}$ 
(equivalently, it is the localizing subcategory as a triangulated generated by $\Sigma^\infty_{\P^1} X_+$ for varying $X$, which means the smallest full triangulated subcategory containing those objects and closed under direct sums). See also \cite[\S 4]{logDMcras} for a quick recollection. 

For every integer $i\in \Z$,
let $\Sigma_{\P^1}^i \logSH_\tau(S,\Lambda)^\eff\subset \logSH_\tau(S,\Lambda)$ be the 
full subcategory of $\logSH_\tau(S,\Lambda)$ generated under colimits by $\Sigma_{\P^1}^iE$ for $E\in \logSH_\tau(S,\Lambda)^\eff$.
These categories form an exhaustive filtration of $\logSH_\tau(S, \Lambda)$ by subcategories 
\[
\cdots
\subset
\Sigma_{\P^1}^{i+1}\logSH_\tau(S,\Lambda)^\eff
\subset
\Sigma_{\P^1}^i\logSH_\tau(S,\Lambda)^\eff
\subset
\cdots
\subset
\logSH_\tau(S,\Lambda).
\]

 \begin{df}
We will write $\logSH_\tau(S, \Lambda)_{\geq i}^{\eff}$ for 
 $\Sigma_{\P^1}^i \logSH_\tau(S, \Lambda)^{\eff}$ and call it the $\infty$-category of \emph{$i$-effective spectra}.

 Similarly, let $\logSH_\tau(S,\Lambda)_{\leq i}^\eff$ be the $\infty$-category of \emph{$i$-coeffective spectra},
that is the full subcategory of $\logSH_\tau(S,\Lambda)$ spanned by those objects $\cG$ such that $\Hom_{\logSH_\tau(S,\Lambda)}(\cF,\cG)\simeq 0$ for every $\cF\in \logSH_\tau(S,\Lambda)_{\geq i+1}^\eff$.
\end{df}
The inclusion $\iota_i\colon \Sigma_{\P^1}^i\logSH_\tau(S,\Lambda)^\eff\subset \logSH_\tau(S,\Lambda)$ preserves colimits and hence admits a right adjoint $\mathrm{r}_i$.
Let $\fil_i^\tau\to \id$ be the counit of this adjunction pair,
and let $\fil^{i-1}_\tau$ be the cofiber of this natural transformation.
Observe that the essential image of $\fil_\tau^i$ is in $\Sigma_{\P^1}^i \logSH_\tau(S,\Lambda)^\eff$.
We set $\slice_i^\tau:=\cofib(\fil_{i+1}^\tau\to \fil_i^\tau)$.
We omit $\tau$ in $\fil_i^\tau$, $\fil^i_\tau$, and $\slice_i^\tau$ when $\tau=\sNis$.

We also remark that the inclusion $\iota_i$ preserves compact objects. 
Thus, its right adjoint $\mathrm{r}_i$ preserves filtered colimits \cite[Proposition 5.5.7.2]{HTT}, and hence the composite $\fil_i=\iota_i\mathrm{r}_i$ preserves filtered colimits too.
\end{const}
\begin{rmk}
We note that $\logSH(S)^\eff$ is a stable symmetric monoidal $\infty$-subcategory of $\logSH(S)$.
\end{rmk}

We also consider the log version of the very effective slice filtration following
Spitzweck-{\O}stv{\ae}r \cite[\S 5]{SO:twistedK}.

\begin{const}\label{constr.veff}Let 
 $\logSH_\tau(S, \Lambda)^\veff$ be the smallest full subcategory of $\logSH_\tau(S, \Lambda)$ containing $\Sigma_{\P^1}^\infty X_+$, closed under colimits and closed under extensions in the sense that if $X\to Y \to Z$ is a cofiber sequence in $\logSH_\tau(S, \Lambda)$ with $X$ and $Z$ in $\logSH_\tau(S, \Lambda)^{\rm veff}$, then also $Y\in \logSH_\tau(S, \Lambda)^{\rm veff}$. 
   We note that $\logSH_\tau(S, \Lambda)^{\rm veff}$ is a symmetric monoidal $\infty$-subcategory of $\logSH_\tau(S,\Lambda)$, but unlike $\logSH_\tau(S, \Lambda)^\eff$ it is not stable.

We can consider the colimit-preserving inclusions (we omit $\tau$ and $\Lambda$ for simplicity)
\[ \Sigma_{\P^1}^{i+1}\logSH(S)^{\rm veff}  \subset \Sigma_{\P^1}^i \logSH(S)^{\rm veff} \subset \Sigma_{\P^1}^{i-1} \logSH(S)^{\rm veff}  \subset \cdots \subset \logSH(S)\]
giving rise via their right adjoints to functorial filtrations
\[
\tilde{\fil}_{i+1}^\tau E  \to \tilde{\fil}_i^\tau E \to \tilde{\fil}_{i-1}^\tau E \to \cdots \to E
\]
where each $\tilde{\fil}_i^\tau$ is the right adjoint to the inclusion $\Sigma_{\P^1}^i\logSH_\tau(S, \Lambda)^{\rm veff} \subset \logSH_\tau(S, \Lambda)$.
The very effective slice functors $\tilde{\slice}_i^\tau\colon \logSH_\tau(S, \Lambda) \to \logSH_\tau(S, \Lambda)$ are then defined by the cofiber sequence
\[ \tilde{\fil}_{i+1}^\tau \to \tilde{\fil}_i^\tau \to \tilde{\slice}_i^\tau. \]
Note that one can effectively ``compute'' the $i$th slice as $\tilde{\slice}_{i}^\tau E \simeq \Sigma_{\P^1}^i \tilde{\slice}_{0}^\tau E$ for every motivic spectrum $E$ that satisfies Bott periodicity $\Sigma_{\P^1} E\simeq E$. 
As before, we omit $\tau$ in $\fil_i^\tau$ and $\fil^i_\tau$ when $\tau=\sNis$.

\begin{df} We will write $\logSH_\tau(S, \Lambda)_{\geq i}^{\veff}$ for $\Sigma_{\P^1}^i \logSH_\tau(S, \Lambda)^{\veff}$.
We define for every integer $i$ the category $\logSH_\tau(S, \Lambda)^{\veff}_{\leq 0}$ to be the category of \emph{co-very effective} spectra, that is the full subcategory of $\logSH_\tau(S, \Lambda)$ that is spanned by those objects $\cG$ such that $\Hom_{\logSH_\tau(S,\Lambda)}(\cF, \cG)\simeq 0$ for every $\cF\in  \logSH_\tau(S, \Lambda)^{\veff}_{\geq 1}$. We similarly write $\logSH_\tau(S, \Lambda)_{\leq i}^{\veff}$ for the right orthogonal subcategory to $ \logSH_\tau(S, \Lambda)^{\veff}_{\geq i+1}$.
\end{df}
\end{const}

\begin{rmk}
For every integer $i$,
observe that we have the obvious inclusions
\[
\logSH(S)_{\geq i}^\veff
\subset
\logSH(S)_{\geq i}^\eff
\text{ and }
\logSH(S)_{\leq i}^\eff
\subset
\logSH(S)_{\leq i}^\veff.
\]
\end{rmk}
\begin{rmk}Assume that $S=\Spec(k)$ is a field and that $\tau=\mathrm{dNis}$.
    As in $\mathbb{A}^1$-motivic homotopy theory, $\logSH_\tau(k,\Lambda)^{\rm veff}$ is the connective part of $t$-structures  on $\logSH_\tau(S,\Lambda)$ and on $\logSH_\tau(S,\Lambda)^{\rm eff}$ by \cite[Proposition 1.4.4.11]{HA}. Note that the analog of Morel's connectivity theorem in the log setting \cite[Theorem 3.2]{Binda_Merici}, together with semi-purity \cite[Theorem 4.4]{Binda_Merici}, gives that this is exactly the non-negative part of the standard homotopy $t$-structure (see \cite[Theorem 5.2.3]{MR2175638}).
\end{rmk}

\begin{rmk}
    A similar construction can be performed in other motivic settings. 
    For example, one could consider in the category of 
    $\P^1$-spectra $\MS_S$ introduced by Annala-Iwasa \cite{AI} the subcategory $\mathrm{MS}^{\rm veff}_S$ generated by $\Sigma_{\P^1}^\infty X_+$ for $X\in \Sm/S$ and closed under colimits and extensions. Again by \cite[Proposition 1.4.4.11]{HA}, 
    this is the connective part of a  $t$-structure on $\MS_S$ (that we can call the 
    homotopy $t$-structure, if the expected connection to homotopy sheaves holds). We can consider the colimit-preserving inclusions
    \[
   \Sigma_{\P^1}^{i+1}  {\MS}^\veff_S \subset  \Sigma_{\P^1}^i  \MS^\veff_S \subset \cdots \subset \MS_S
    \]
    giving rise via their right adjoints to functorial filtrations
    \[ \tilde{\fil}_{i+1} E \to \tilde{\fil}_i E \to \cdots \to E  \]
    for any spectrum $E$.
\end{rmk}
Recall from \cite[Construction 4.0.18]{BPO2} that there is a symmetric monoidal functor 
\[\lambda_\sharp\colon \MS_S \to \logSH(S)\]
induced by the functor $\lambda\colon \Sm/S \to \SmlSm/S$ given by $\lambda(X) := X$ (seen as log scheme with trivial log structure). The following is immediate from the definitions. 
\begin{prop}Let $S$ be any quasi-compact quasi-separated scheme. Then for every integer $n$ we have the inclusion $\lambda_{\sharp} \Sigma_{\P^1}^n\MS_S^\veff \subset \Sigma_{\P^1}^n \logSH(S)^\veff$.  
\end{prop}

We record some elementary properties of the slice filtration.
\begin{lem}
\label{slice.4}
Let $\tau'$ be a topology on $\SmlSm/S$ finer than $\tau$.
Then for every integer $i$,
the $\tau'$-localization functor
\[
L_{\tau'}
\colon
\logSH_\tau(S,\Lambda)
\to
\logSH_{\tau'}(S,\Lambda)
\]
sends $\logSH_\tau(S,\Lambda)_{\geq i}^\eff$ into $\logSH_{\tau'}(S,\Lambda)_{\geq i}^\eff$,
and its right adjoint 
\[
\iota
\colon
\logSH_{\tau'}(S,\Lambda)
\hookrightarrow
\logSH_\tau(S,\Lambda)
\]
sends $\logSH_{\tau'}(S,\Lambda)_{\leq i}^\eff$ into $\logSH_\tau(S,\Lambda)_{\leq i}^\eff$.
A similar result holds for the very effective version, too.
\end{lem}
\begin{proof}
We only give proof for the effective version.
The statement about the left adjoint is a consequence of the fact that $L_{\tau'}$ preserves colimits and sends $\Sigma^{2n,n}\Sigma^\infty X_+$ to  $\Sigma^{2n,n}\Sigma^\infty X_+$ for every integer $n$.
Use the natural isomorphism
\[
\Hom_{\logSH_{\tau}(S,\Lambda)}
(\cF,\iota \cG)
\cong
\Hom_{\logSH_{\tau'}(S,\Lambda)}(L_{\tau'}\cF,\cG)
\]
for $\cF\in \logSH_{\tau}(S,\Lambda)_{\geq i+1}^\eff$ and $\cG\in \logSH_{\tau'}(S,\Lambda)_{\leq i}^\eff$ to deduce the statement about the right adjoint.
\end{proof}
\begin{lem}
\label{slice.3}
Let $\Lambda\to \Lambda'$ be a map of $\E_\infty$-rings.
Then the base-change functor
\[
- \otimes_\Lambda \Lambda'
\colon
\logSH_\tau(S,\Lambda)
\to
\logSH_\tau(S,\Lambda')
\]
sends $\logSH_\tau(S,\Lambda)_{\geq i}^\eff$ into $\logSH_\tau(S,\Lambda')_{\geq i}^\eff$,
and the restriction functor
\[
\logSH_\tau(S,\Lambda')
\to
\logSH_\tau(S,\Lambda)
\]
sends $\logSH_\tau(S,\Lambda')_{\leq i}^\eff$ to $\logSH_\tau(S,\Lambda)_{\leq i}^\eff$.
A similar result holds for the very effective version, too.
\end{lem}
\begin{proof}
This is a consequence of the fact that $-\otimes_\Lambda \Lambda'$ preserves colimits and sends $\Sigma^{2n,n}\Sigma^\infty X_+$ to  $\Sigma^{2n,n}\Sigma^\infty X_+$ for every $X\in \SmlSm/S$ and integer $n$.
\end{proof}
In particular, the restriction functor for the morphism $\Sph \to H\Z$
\[ \logDA(S)= \logSH(S, H\Z) \to \logSH(S)\]
sends $\logDA(S)_{\leq i}^\eff$ to $\logSH(S)_{\leq i}^\eff$.

The following easy Lemma will be useful for the main theorem later in the text.
\begin{lem}
\label{slice.2}
If $\cF\in \logSH_\tau(S,\Lambda)$ admits a filtration such that $\Fil^i\cF\in \logSH_\tau(S,\Lambda)_{\leq i}^\eff$ for every integer $i$,
then, there exists a unique natural morphism of filtered objects
\[
\fil_\bullet^\tau \cF
\to
\Fil_\bullet \cF
\]
whose morphism of underlying objects is $\id \colon \cF\to \cF$.
Similarly, if $\cF \in \logSH_\tau(S, \Lambda)$ admits a filtration such that $\Fil^i  \cF \in   \logSH_\tau(S, \Lambda)^{\veff}_{\leq i}$ for every $i$, then there exists a unique natural morphism of filtered objects 
\[
\tilde{\fil}_\bullet^\tau\cF
\to
\Fil_\bullet \cF
\]
whose morphism of underlying objects is $\id \colon \cF\to \cF$.
\end{lem}
\begin{proof}
We only give proof for the effective version. 
Consider the naturally induced maps
\[
\fil_\tau^i\cF
\xrightarrow{}
\fil_\tau^i(\Fil^i \cF)
\xrightarrow{\simeq}
\Fil^i \cF,
\]
where the first map is obtained by functoriality and the second map is the inverse of the counit of the adjunction, which is an equivalence since  $\Fil^i \cF\in \logSH_\tau(S,\Lambda)_{\leq i}^\eff$ for every integer $i$.
By taking $\fib(\cF\to (-))$,
we get the desired morphism
\(
\fil_\bullet^\tau\cF
\to
\Fil_\bullet\cF
\).
For its uniqueness,
consider the induced exact sequence
\begin{align*}
&
\Hom_{\logSH_\tau(S,\Lambda)}(\Sigma^{1,0}\fil_{i+1}^\tau\cF,\Fil^i \cF)
\to 
\Hom_{\logSH_\tau(S,\Lambda)}(\fil_\tau^i\cF,\Fil^i \cF)
\\
\to &
\Hom_{\logSH_\tau(S,\Lambda)}(\cF,\Fil^i \cF)
\to
\Hom_{\logSH_\tau(S,\Lambda)}(\fil_{i+1}^\tau\cF,\Fil^i \cF).
\end{align*}
The first and fourth terms are $0$ since $\Sigma^{1,0}$ sends $\logSH_\tau(S,\Lambda)_{\geq i+1}^\eff$ into itself (note that this works verbatim for $\logSH_\tau(S,\Lambda)_{\geq i+1}^{\rm veff}$ as well, since it is closed under  colimits, hence in particular under positive suspensions).
This establishes the uniqueness of $f^\bullet\cF\to \Fil^\bullet \cF$ and hence the uniqueness of $\fil_\bullet^\tau \cF\to \Fil_\bullet\cF$.
\end{proof}

\section{Slice filtration on K-theory}

One of the fundamental results in $\A^1$-homotopy theory is the equivalence 
$\slice_i \KGL\simeq \Sigma^{2i,i}\MZ$ in $\SH(k)$,
where $k$ is a perfect field, due to Voevodsky \cite{VoevZeroslice} in characteristic zero and Levine \cite{MR2365658} in general (this is now known more generally, 
e.g., over Dedekind schemes by \cite{BachmannJEMS}, \cite[Theorem 2.19]{zbMATH07003144}. 
 The effective and the very effective slices of $\KGL$ coincide, 
see \cite[Lemma 2.4]{AnRoOest}, 
so that, in particular, the motivic Eilenberg-MacLane spectrum $\bM\Z$ is very effective. 

In this section,
we promote this to $\logSH(k)$, assuming resolution of singularities. 
\begin{prop}
\label{K.1}
Let $S$ be a quasi-compact quasi-separated scheme.
Then, for every integer $i$,
we have the inclusion
\[
\omega^* \SH(S)_{\leq i}^\eff
\subset
\logSH(S)_{\leq i}^\eff
\]
A similar result holds for the very effective version, too.
\end{prop}
\begin{proof}
By adjunction,
it suffices to show $\omega_\sharp\logSH(S)_{\geq i}^\eff\subset \SH(S)_{\geq i}^\eff$ and 
$\omega_\sharp\logSH(S)_{\geq i}^\veff\subset \SH(S)_{\geq i}^\veff$.
This holds since $\omega_\sharp$ preserves colimits and $\omega_\sharp \Sigma^{n,0} \Sigma^\infty Y_+\simeq \Sigma^{n,0} \Sigma^\infty (Y-\partial Y)_+$ for $Y\in \SmlSm/S$ and integer $n$.
\end{proof}

\begin{prop}
\label{K.2}
Let $k$ be a perfect field admitting resolution of singularities.
Then, for every integer $i$,
we have the inclusion
\[
\omega^* \SH(k)_{\geq i}^\eff
\subset
\logSH(k)_{\geq i}^\eff.
\]
A similar result holds for the very effective version, too.
\end{prop}
\begin{proof}
For $X\in \Sm/k$,
there exists proper $Y\in \SmlSm/k$ such that $Y-\partial Y\cong X$ by resolution of singularities.
Due to \cite[Theorem 1.1(2)]{Par},
we have an equivalence
$\omega^* \Sigma^{n,0} \Sigma^\infty X_+ \simeq \Sigma^{n,0} \Sigma^\infty Y_+$ for every integer $n$.
To conclude,
observe that $\omega^*$ preserves colimits by \cite[Theorem 1.1(4)]{Par}.
\end{proof}

\begin{prop}
\label{K.3}
Let $k$ be a perfect field admitting resolution of singularities.
For $\cF\in \SH(k)$ and integer $i$,
there is a natural equivalence
\[
\omega^*\fil_i\cF
\simeq
\fil_i\omega^* \cF
\]
in $\logSH(k)$.
A similar result holds for the very effective version, too.
\end{prop}
\begin{proof}
We have the natural fiber sequence
\[
\fil_i\omega^*\fil_i\cF
\to
\fil_i\omega^*\cF
\to
\fil_i\omega^*\fil^{i-1}\cF
\]
in $\logSH(k)$.
Proposition \ref{K.2} implies $\fil_i\omega^*\fil_i\cF\simeq \omega^*\fil_i\cF$,
and Proposition \ref{K.1} implies
$\fil_i\omega^*\fil^{i-1}\cF \simeq 0$.
From these,
we obtain the desired equivalence.
\end{proof}
Let now $\KGL$ denote the motivic $\P^1$-spectrum representing algebraic $K$-theory in $\logSH(k)$ constructed in \cite[Definition 6.5.6]{BPO2} (denoted $\mathbf{logKGL}$ in \emph{loc.~cit.}), and let $\MZ \in \logSH(k)$ denote the motivic $\P^1$-spectrum representing integral motivic cohomology constructed in \cite[Definition 7.7.5]{BPO2}, denoted $\mathbf{logM}\mathbb{Z}$ in \emph{loc.~cit.}). Since  $\KGL$ and $\MZ$ satisfy the property that  $\MZ \simeq \omega^* \MZ$  when $k$ admits resolution of singularities, and $\KGL \simeq \omega^* \KGL$ in general, we may unambiguously drop the prefix $\mathbf{log}$ from the notation.

\begin{thm}
\label{K.4}
Let $k$ be a perfect field admitting resolution of singularities.
Then the slice filtration and the very effective slice filtration on $\KGL\in \logSH(k)$ agree
and are complete and exhaustive.
Furthermore,
there are natural equivalences
\[
\slice_i \KGL
\simeq
\widetilde{\slice}_i \KGL
\simeq
\MZ(i)[2i]
\]
in $\logSH(k)$ for every integer $i$, where $\MZ(i)[2i]$ denotes the suspension $\Sigma^{2i,i}\MZ$.
\end{thm}
\begin{proof}
We have equivalences
\[
\slice_i \KGL
\simeq^{(1)}
\slice_i \omega^*\KGL
\simeq^{(2)}
\omega^* \slice_i\KGL
\simeq^{(3)}
\omega^* \Sigma^{2i,i}\MZ
\simeq^{(4)}
\Sigma^{2i,i}\MZ
\]
in $\logSH(k)$,
where (1) is due to \cite[Definition 6.5.6]{BPO2},
(2) is due to Proposition \ref{K.3},
(3) is due to \cite[Theorems 6.4.2, 9.0.3]{MR2365658},
and (4) is due to \cite[Proposition 7.7.6]{BPO2}.

We also have $\fil_i \KGL \simeq \omega^*\fil_i\KGL$.
To show that the slice filtration on $\KGL\in \logSH(k)$ is complete and exhaustive,
it suffices to show that the slice filtration on $\KGL\in \SH(k)$ is complete and exhaustive since $\omega^*$ preserves colimits and limits.
The fact that $\KGL$ is slice complete can be found, e.g., in \cite[Lemma 3.11]{RSO_hopf}.

For the very effective slice of $\KGL\in \logSH(k)$,
use \cite[Lemma 2.4]{AnRoOest} and argue similarly as above.
\end{proof}

\section{HKR filtration on logarithmic Hochschild homology}

Throughout this section, we fix 
$(R,P)$ a pre-log ring,
and $S$ a quasi-compact quasi-separated scheme over $\Spec(R)$. 
In practice, 
the reader can ignore $P$ in most statements and assume that $R$ is always considered with a trivial log structure. 

Our goal is to construct the HKR filtration on $\bHH(-/R)$ and Beilinson filtrations on $\bHC^-(-/R)$ and $\bHP(-/R)$ in $\logDA(S)$,
which are the $\P^1$-stabilized versions of the corresponding filtrations on $\HH(-/R)$, $\HC^-(-/R)$, and $\HP(-/R)$, built, in the log setting, in our previous works \cite{BLPO} and \cite{BLPO2}.
We also explore their fundamental properties.

\begin{df}For a map of pre-log rings $(R,P)\to (A,M)$ and integer $i$,
we set
\[
\L\Omega_{(A,M)/(R,P)}^i:=\wedge_A^i \L_{(A,M)/(R,P)}.
\]
where $\L_{(A,M)/(R,P)}$ is Gabber's logarithmic cotangent complex \cite[\S 8]{OlsCot}, and $\wedge^i_A(-)$ is the $i$th derived exterior power. The \emph{logarithmic derived de Rham cohomology} of $(R,P)\to (A,M)$ is defined as the   total complex
\[
\L\Omega_{(A,M)/(R,P)}
:=
\Tot(\L\Omega_{(A,M)/(R,P)}^0 \to \L\Omega_{(A,M)/(R,P)}^1 \to \cdots).
\]
This admits the Hodge filtration given by
\[
\L\Omega_{(A,M)/(R,P)}^{\geq i}
:=
\Tot(0\to \cdots \to 0 \to \L\Omega_{(A,M)/(R,P)}^i \to \L\Omega_{(A,M)/(R,P)}^{i+1} \to \cdots)
\]
where $\L\Omega_{(A,M)/(R,P)}^i$ is in cohomological degree $i$  for $i\geq 0$ and $\L\Omega_{(A,M)/(R,P)}^{\geq i}
:=\L\Omega_{(A,M)/(R,P)}$ for $i<0$.
Let $\widehat{\L\Omega}_{(A,M)/(R,P)}$ be the completion of $\L\Omega_{(A,M)/(R,P)}$ with respect to the Hodge filtration,
and let $\widehat{\L\Omega}_{X/(R,P)}^{\geq n}$ be the $n$th term of the induced Hodge filtration of $\widehat{\L\Omega}_{(A,M)/(R,P)}$.
\end{df}
\begin{rmk}
By Zariski descent,
we can also define the sheaf of complexes $\L\Omega_{X/(R,P)}^i$, $\L\Omega_{X/(R,P)}$, $\L\Omega_{X/(R,P)}^{\geq i}$, $\widehat{\L\Omega}_{X/(R,P)}$, and $\widehat{\L\Omega}_{X/(R,P)}^{\geq i}$ on $X_{\Zar}$ for every log scheme $X$ over $(R,P)$.
\end{rmk}
We now list some facts that we will use in the paper.
\begin{enumerate}
\item For any map of pre-log rings $(R, P)\to (A,M)$, the logarithmic Hochschild homology 
$\HH((A,M)/(R,P))$ of \cite[Definition 5.3]{BLPO} admits a separated, descending filtration with
graded pieces $\wedge_A^i \L_{(A,M)/(R,P)}[i]$. 
This is obtained by Kan extension of the Postnikov filtration in the polynomial case. 
We shall refer to this as the log HKR filtration, see \cite[Theorem 1.1]{BLPO}.

\item The negative cyclic homology and the periodic homology $\HC^-((A,M)/(R,P))$ and $\HP((A,M)/(R,P))$ admit the complete exhaustive \emph{Beilinson filtrations} \[\Fil^{\rB}_{\geq \bullet}\HC^-((A,M)/(R,P)), \quad \Fil^{\rB}_{\geq \bullet}\HP((A,M)/(R,P))\] with graded pieces $\widehat{\L \Omega}^{\geq n}_{(A,M)/(R,P)}[2n]$ and $\widehat{\L \Omega}_{(A,M)/(R,P)}[2n]$.
The non-log case is due to Antieau \cite[Theorem 1.1]{Ant19}, while the log case is \cite[Theorem 1.1]{BLPO2}. 
If $(A,M)\in \Poly_{(R,P)}$,
then
\begin{gather*}
\Fil^{\rB}_i\HC^-((A,M)/(R,P))
:=
\tau_{\geq 2i}^\rB\tau_{\geq \bullet} \HC^-((A,M)/(R,P)),
\\
\Fil^{\rB}_i\HP((A,M)/(R,P))
:=
\tau_{\geq 2i}^\rB\tau_{\geq \bullet} \HP((A,M)/(R,P)),
\end{gather*}
where $\tau_{\geq i}^{\rB}$ denotes the truncation functor for the Beilinson $t$-structure on the filtered derived category $\DF(A)$ \cite[Definition 2.2]{Ant19}.
We obtain the filtrations for general $(A,M)$ by left Kan extension and regarding $\HC^-$ and $\HP$ as bicomplete bifiltered complexes,
see \cite[Remark 4.4]{Ant19}, or \cite[2.2]{BLPO2}.
\end{enumerate}

\begin{const}
\label{HKR.7}
Let $X$ be a log scheme.
The inclusion functor from the category of line bundles over $\ul{X}$ to the category of vector bundles over $\ul{X}$ yields the canonical morphism
\[
c_1^{\Kth}
\colon
R\Gamma_{\Zar}(\ul{X},\G_m)[1]
\to
\Kth(\ul{X}).
\]
The \emph{first Chern class for $\TC$} is the composite morphism
\[
c_1^{\TC}
\colon
R\Gamma_{\Zar}(\ul{X},\G_m)[1]
\xrightarrow{c_1^{\Kth}}
\Kth(\ul{X})
\xrightarrow{\Tr}
\TC(\ul{X})
\xrightarrow{p^*}
\TC(X),
\]
where $p\colon X\to \ul{X}$ is the morphism removing the log structure.
We similarly obtain the morphisms $c_1^{\THH}$, $c_1^{\TC^-}$, $c_1^{\TP}$, $c_1^{\HH}$, $c_1^{\HC^-}$, and $c_1^{\HP}$ to $\THH(X)$, $\TC^-(X)$, $\TP(X)$, $\HH(X)$, $\HC^-(X)$, and $\HP(X)$.

\begin{rmk}We will set the following notation for convenience when dealing with projective bundle formulas below:
For a morphism of commutative ring spectra $f\colon \cF\to \cG$ with $c\in \pi_0(\cG)$,
we will often write $fc$ (or simply $c$ if $f$ is clear from the context) for the composite morphism of spectra
\begin{equation}
\label{HKR.7.1}
fc
\colon
\cF
\xrightarrow{\simeq}
\cF\otimes_\Sph \Sph
\xrightarrow{f\otimes c}
\cG\otimes_\Sph \cG
\xrightarrow{\mu}
\cG,
\end{equation}
where $\mu$ is the multiplication.
We also consider the unit $1\in \pi_0(\cG)$.
\end{rmk}

Recall the projective bundle formula for $\Kth(\ul{X}\times \P^1)$:
The morphism of spectra
\[
(1,c_1^\Kth(\cO(1)))
\colon
\Kth(\ul{X})\oplus \Kth(\ul{X})
\to
\Kth(\ul{X}\times \P^1)
\]
is an equivalence,
where $\cO(1)\in \Pic(\ul{X}\times \P^1)\cong \pi_0(R\Gamma_{\Zar}(\ul{X},\G_m)[1])$,
and $c_1^\Kth(\cO(1))\in \pi_0(K(\ul{X}\times \P^1))$ (implicitly we are writing $1$ for the pull-back map along $\pi\colon \underline{X}\times \P^1 \to \underline{X}$).
Due to \cite[Theorem 1.5]{BM12},
we similarly have the projective bundle formula for $\TC(\ul{X}\times \P^1)$:
The morphism of spectra
\begin{equation}
\label{HKR.7.2}
(1,c_1^\TC(\cO(1)))
\colon
\TC(\ul{X})\oplus \TC(\ul{X})
\to
\TC(\ul{X}\times \P^1)
\end{equation}
is an equivalence.
Analogous results hold for $\THH$, $\TC^-$, $\TP$, $\HH$, $\HC^-$, and $\HP$.
\end{const}

\begin{prop}
\label{BMS.16}
There is an $S^1$-equivariant equivalence of filtered complexes
\begin{equation}
\label{BMS.16.1}
\Z\oplus (\Z[1])\{-1\}
\simeq
\HH(\P^1/\Z),
\end{equation}
where the filtrations on $\Z$ and $\Z[1]$ are the Postnikov filtrations.
Here, the notation $(\Z[1])\{-1\}$ means
\[
\Fil_i ((\Z[1])\{-1\})
=
\left\{
\begin{array}{ll}
\Z & \text{if $i\leq 1$},
\\
0 & \text{otherwise}.
\end{array}
\right.
\]
\end{prop}
\begin{proof}
Every $S^1$-action on $\Z\oplus \Z$ is trivial,
so it suffices to show the claim after forgetting the $S^1$-actions. 
We have the equivalences of filtered complexes
\begin{gather*}
\Z[x]\oplus \Z[x]dx[1]
\simeq
\HH(\Z[x]),
\\
\Z[x^{-1}]\oplus \Z[x^{-1}]d(x^{-1})[1]
\simeq
\HH(\Z[x^{-1}]),
\\
\Z[x,x^{-1}]\oplus \Z[x,x^{-1}]dx[1]
\simeq
\HH(\Z[x,x^{-1}]),
\end{gather*}
where the filtrations are the Postnikov filtrations on both sides.
Consider the standard cover on $\P^1$,
and use the equivalences of complexes
\begin{gather*}
\Z
\simeq
\fib(\Z[x]\oplus \Z[x^{-1}]\to \Z[x,x^{-1}])
\\
\Z[-1]
\simeq
\fib(\Z[x]dx\oplus \Z[x^{-1}]d(x^{-1})\to \Z[x,x^{-1}]dx)
\end{gather*}
to obtain the desired equivalence.
\end{proof}

\begin{prop}
\label{HKR.8}
Let $X$ be a log scheme.
Then the morphism of spectra
\begin{equation}
\label{HKR.8.1}
(1,c_1^\TC(\cO(1)))
\colon
\TC(X)\oplus \TC(X)
\to
\TC(X\times \P^1)
\end{equation}
is an equivalence,
and this is obtained by applying $\TC(X)\otimes_\Z-$ to \eqref{BMS.16.1}.
We have similar results for $\THH$, $\TC^-$, $\TP$, $\HH$, $\HC^-$, and $\HP$.
\end{prop}
\begin{proof}
We focus on $\TC$ since the proofs are similar. By base-change along $\TC(\ul{X}) \to \TC(X)$ from \eqref{HKR.7.2}
we get \eqref{HKR.8.1}.
Hence we reduce to the case when $X$ has a trivial log structure.

By \cite[Theorem 9.1 and p.\ 1102]{BM12},
we have the commutative diagram of spectra
\[
\begin{tikzcd}
\Kth(X)[1]\ar[r,"\tau_\Kth"]\ar[d,"\Tr"']&
\Kth(X\times \G_m)\ar[d,"\Tr"]\ar[r,"\partial"]&
\Kth(X\times \P^1)[1]\ar[d,"\Tr"]
\\
\TC(X)[1]\ar[r,"\tau_\TC"]&
\TC(X\times \G_m)\ar[r,"\partial"]&
\TC(X\times \P^1)[1],
\end{tikzcd}
\]
where $\tau_\Kth$ and $\tau_\TC$ are obtained by the Bass functor structures on $\Kth$ and $\TC$,
and the boundary morphisms $\partial$ are obtained by the standard cover of $\P^1$.
The composite $\partial \circ \tau_K$ sends $1\in \pi_0\Kth(X)$ to $c_1^\Kth(\cO(1))\in \pi_0\Kth(X\times \P^1)$,
so the composite $\partial \circ \tau_{\TC}$ sends $1\in \pi_0\TC(X)$ to $c_1^\TC(\cO(1))\in \pi_0\TC(X\times \P^1)$.
It follows that the morphism
\[
(1,\partial \circ \tau_\TC)
\colon
\TC(X)\oplus \TC(X)
\to
\TC(X\times \P^1)
\]
agrees with $(1,c_1^\TC(\cO(1)))$.

By \cite[pp.\ 1102--1103]{BM12},
we see that $\tau_\THH\colon \THH(X)[1]\to \THH(X\times \G_m)$ is obtained by applying $\THH(X)\otimes_{\Z}-$ to $\tau_\HH\colon \HH(\Z)[1]\to\HH(\Z[x,x^{-1}])$ that sends $1\in \pi_0\HH(\Z)$ to $dx/x\in \Omega_{\Z[x,x^{-1}]/\Z}\cong \pi_1\HH(\Z[x,x^{-1}])$.
Furthermore,
$\tau_\TC$ is obtained by applying $\TC$ to $\tau_\THH$.
Compare this with \eqref{BMS.16.1} to conclude.
\end{proof}

\begin{prop}
\label{HKR.1}
Let $(R,P)$ be a pre-log ring.
For  a quasi-compact quasi-separated log scheme $X$ over $(R,P)$ and $i\in \Z\cup \{-\infty\}$,
the projective bundle formula for $\HH(X\times \P^1/(R,P))$ restricts to natural equivalences of complexes
\begin{gather*}
\Fil^{\HKR}_i\HH(X\times \P^1/(R,P))
\simeq
\Fil^{\HKR}_i\HH(X/(R,P))
\oplus
\Fil^{\HKR}_{i-1}\HH(X/(R,P)),
\\
\Fil^{\rB}_i\HC^-(X\times \P^1/(R,P))
\simeq
\Fil^{\rB}_i\HC^-(X/(R,P))
\oplus
\Fil^{\rB}_{i-1}\HC^-(X/(R,P)),
\\
\Fil^{\rB}_i\HP(X\times \P^1/(R,P))
\simeq
\Fil^{\rB}_i\HP(X/(R,P))
\oplus
\Fil^{\rB}_{i-1}\HP(X/(R,P)).
\end{gather*}
\end{prop}
\begin{proof}
We refer to \cite[\S 2.2]{BLPO2} for the category $\Poly_{(R,P)}$ of polynomial $(R,P)$-algebras.
By Zariski descent and left Kan extension,
we reduce to the case when $X:=\Spec(A,M)$ with $(A,M)\in \Poly_{(R,P)}$.
In this case,
the HKR filtration on $\HH(X/(R,P))$ is the Postnikov filtration.
Hence we have the $S^1$-equivariant equivalence of filtered complexes 
\[
\HH(X\times \P^1/(R,P))
\simeq
\HH(X/(R,P))\otimes \HH(\P^1/\Z).
\]  
Together with Proposition \ref{BMS.16},
we obtain the $S^1$-equivariant decomposition of filtered complexes
\begin{equation}
\label{HKR.1.1}
\HH(X\times \P^1/(R,P))
\simeq
\HH(X/(R,P))
\oplus
(\HH(X/(R,P))[1])\{-1\},
\end{equation}
where the filtrations on $\HH(X/(R,P))$ and $\HH(X/(R,P))[1]$ are the Postnikov filtrations.
This implies the claim for $\HH$.

By applying $(-)^{hS^1}$ on both sides of \eqref{HKR.1.1},
we obtain an equivalence of complexes induced by the homotopy fixed point spectral sequence:
\[
\Fil^{\HKR}_\bullet\HC^-(X\times \P^1/(R,P))
\simeq
\Fil^{\HKR}_\bullet\HC^-(X/(R,P))
\oplus
\Fil^{\HKR}_{\bullet-1}\HC^-(X/(R,P)).
\]
Take $\tau_{\geq 2i}^{\rB}$ on both sides to show the claim for $\HC^-$.
The claim for $\HP$ can be proven similarly.
\end{proof}

\begin{const}
\label{BMS.17} The local-to-global spectral sequence $E_2^{st}=H^s(X, \Omega^t_{X}) \Rightarrow \pi_{t-s}\HH(X)$ for $X=\P^1$, together with the decomposition of $\HH(\P^1/\Z)$ of Proposition \ref{BMS.16}, induces by restriction of  $c_1^\HH(\cO(1))\in \pi_0\HH(\P^1)$ a class
we obtain
\[
c_1^{\Hod}(\cO(1))
\in
H_{\Zar}^1(\P^1,\Omega_{\P^1/\Z}^1),
\]
that agrees with the classical first Chern class in Hodge cohomology. After pulling back,
we obtain
\[
c_1^{\Hod}(\cO(1))
\in
H_{\Zar}^1(X\times \P^1,\L\Omega_{X/(R,P)}^1)
\]
for a log scheme $X$ over a pre-log ring $(R,P)$.  
By descent and reduction to the polynomial case, we have the following $\P^1$-bundle formulas: 
the natural morphisms of complexes
\begin{gather*}
R\Gamma_{\Zar}(X,\L\Omega_{X/(R,P)}^i)
\oplus
R\Gamma_{\Zar}(X,\L\Omega_{X/(R,P)}^{i-1})[-1]
\to
R\Gamma_{\Zar}(X\times \P^1,\L\Omega_{X\times \P^1/(R,P)}^i),
\\
R\Gamma_{\Zar}(X,\widehat{\L\Omega}_{X/(R,P)}^{\geq i})
\oplus
R\Gamma_{\Zar}(X,\widehat{\L\Omega}_{X/(R,P)}^{\geq i-1})[-1]
\to
R\Gamma_{\Zar}(X\times \P^1,\widehat{\L\Omega}_{X\times \P^1/(R,P)}^{\geq i}),
\\
R\Gamma_{\Zar}(X,\widehat{\L\Omega}_{X/(R,P)})
\oplus
R\Gamma_{\Zar}(X,\widehat{\L\Omega}_{X/(R,P)})[-1]
\to
R\Gamma_{\Zar}(X\times \P^1,\widehat{\L\Omega}_{X\times \P^1/(R,P)})
\end{gather*}
induced by $1$ and $c_1^\Hod(\cO(1))$ are equivalences (see proof of \cite[Proposition 2.20]{BLMP}).

If $X$ is quasi-compact and quasi-separated,
 these morphisms are identified with the natural morphisms
\begin{gather*}
\gr^\HKR_i\HH(X/(R,P))
\oplus
\gr^\HKR_{i-1}\HH(X/(R,P))
\simeq
\gr^\HKR_i\HH(X\times \P^1/(R,P)),
\\
\gr^\rB_i\HC^-(X/(R,P))
\oplus
\gr^\rB_{i-1}\HC^-(X/(R,P))
\simeq
\gr^\rB_i\HC^-(X\times \P^1/(R,P)),
\\
\gr^\rB_i\HP(X/(R,P))
\oplus
\gr^\rB_{i-1}\HP(X/(R,P))
\simeq
\gr^\rB_i\HP(X\times \P^1/(R,P))
\end{gather*}
obtained by Proposition \ref{HKR.1}.
Indeed,
this can be shown for $\HH$ by comparing \eqref{HKR.7.1} with the composite map
\begin{align*}
& R\Gamma_{\Zar}(X,\L\Omega_{X/(R,P)}^{i-1})
\otimes_\Z \Z[-1]
\\
\to &
R\Gamma_{\Zar}(X\times \P^1,\L\Omega_{X\times \P^1/(R,P)}^{i-1})
\otimes_\Z R\Gamma_{\Zar}(X\times \P^1,\L\Omega_{X\times \P^1/(R,P)}^1)
\\
\to &
R\Gamma_{\Zar}(X\times \P^1,\L\Omega_{X\times \P^1/(R,P)}^i).
\end{align*}
A similar argument proves the claim for $\HC^-$ and $\HP$.
\end{const}

\begin{prop}
\label{HKR.6}
Let $(R,P)$ be a pre-log ring.
Then the presheaves of complexes 
\[\Fil^\HKR_i\HH(-/(R,P)),
\text{ }
\Fil^\rB_i\HC^-(-/(R,P)),\text{ and }\quad  \Fil^\rB_i\HP(-/(R,P))\]
on log schemes over $(R,P)$ are $(\P^n,\P^{n-1})$-invariant for $n>0$ and $i\in \Z\cup \{-\infty\}$.
Moreover,
the presheaves of complexes $\L\Omega_{/(R,P)}^i$, $\widehat{\L\Omega}_{/(R,P)}^{\geq i}$, and $\widehat{\L\Omega}_{/(R,P)}$ on the category of log schemes over $(R,P)$ given by
\[
X
\mapsto
R\Gamma_\Zar(X,\L\Omega_{X/(R,P)}^i),
\text{ }
R\Gamma_\Zar(X,\widehat{\L\Omega}_{X/(R,P)}^{\geq i}),
\text{ }
R\Gamma_\Zar(X,\widehat{\L\Omega}_{X/(R,P)})
\]
are $(\P^n,\P^{n-1})$-invariant for $n>0$ and $i\in \Z$.
\end{prop}
\begin{proof} This is basically discussed in \cite[\S 7, 8]{BLPO}, forgetting the filtrations.
Let $X$ be a log scheme over $(R,P)$.
Using the completeness and exhaustiveness of the filtrations,
it suffices to show that the natural map of complexes
\[
\gr^\HKR_i\HH(X/(R,P))
\to
\gr^\HKR_i\HH(X\times (\P^n,\P^{n-1})/(R,P))
\]
and the similar maps for $\gr^\rB_i\HC^-$ and $\gr^\rB_i\HP$ are equivalences.
Furthermore,
we reduce to the case when $X=\Spec(A,M)$ with $(A,M)\in \Poly_{(R,P)}$ by left Kan extension.
Then the claim is a consequence of the $(\P^n,\P^{n-1})$-invariance of Hodge cohomology,
see e.g.\ the proof of \cite[Proposition 9.2.1]{BPO}.
\end{proof}

\begin{df}
\label{HKR.2}
For $X\in \SmlSm/S$ and $i\in \Z\cup \{-\infty\}$,
Proposition \ref{HKR.1} yields a natural equivalence of complexes
\begin{equation}
\label{HKR.2.1}
\Fil^{\HKR}_i\HH(X/R)
\simeq
\Omega_{\P^1}\Fil^{\HKR}_{i+1}\HH(X/R).
\end{equation}
Using this as bonding maps,
Proposition \ref{HKR.6} enables us to define the complete exhaustive filtration
\[
\Fil^{\HKR}_i\bHH(-/R)
:=
(\Fil^{\HKR}_i \HH(-/R),\Fil^{\HKR}_{i+1} \HH(-/R),\ldots)
\]
on $\bHH(-/R)\simeq \Fil^{\HKR}_{-\infty} \bHH(-/R)\in \logDA(S)$, where $\bHH(-/R)$ is the $\P^1$-spectrum representing (log) Hochschild homology constructed in \cite[\S 8]{BLPO}. 
We similarly define the complete exhaustive filtrations
\[
\Fil^{\rB}_i\bHC^-(-/R)
\text{ and }
\Fil^{\rB}_i\bHP(-/R).
\]
\end{df}

\begin{prop}
\label{HKR.4}
For $i\in \Z\cup \{-\infty\}$,
we have a natural equivalence of $\P^1$-spectra
\[
\Fil^{\HKR}_i \bHH(-/R)(1)[2]
\simeq
\Fil^{\HKR}_{i+1} \bHH(-/R).
\]
We have similar equivalences for $\Fil^\rB_i \bHC^-$ and $\Fil^{\rB}_i \bHP$.
\end{prop}
\begin{proof}
For $\Fil^{\HKR}_i \bHH$,
this is a direct consequence of \eqref{HKR.2.1}.
The proofs for the other cases are similar.
\end{proof}

\begin{df}
For $i\in \Z$,
Proposition \ref{HKR.6} enables us to define the $\P^1$-spectra
\begin{gather*}
\L\bOmega_{/R}^i
:=
(\L\Omega_{/R}^i,\L\Omega_{/R}^{i+1}[1],\ldots),
\\
\widehat{\L\bOmega}_{/R}^{\geq i}
:=
(\widehat{\L\Omega}_{/R}^{\geq i},\widehat{\L\Omega}_{/R}^{\geq i+1}[2],\ldots),
\\
\widehat{\L\bOmega}_{/R}
:=
(\widehat{\L\Omega}_{/R},\widehat{\L\Omega}_{/R}[2],\ldots)
\end{gather*}
in $\logDA(S)$ whose bonding maps are given by the projective bundle formulas in Construction \ref{BMS.17}.
If $S\to \Spec(R)$ is smooth,
then we omit $\L$ in $\L\bOmega$.
\end{df}

\begin{prop}
\label{HKR.5}
For $i\in \Z$,
we have natural equivalences 
\begin{gather*}
\gr^\HKR_i \bHH(-/R)
\simeq
\L\bOmega_{/R}^i[i],
\\
\gr^\rB_i \bHC^-(-/R)
\simeq
\widehat{\L\bOmega}_{/R}^{\geq i}[2i],
\\
\gr^\rB_i \bHP(-/R)
\simeq
\widehat{\L\bOmega}_{/R}[2i]
\end{gather*}
in $\logDA(S)$.
\end{prop}
\begin{proof}
The bonding maps defining the $\P^1$-spectra on both sides are compatible, 
as observed in Construction \ref{BMS.17}.
\end{proof}
We are ready to prove our main result on $\HH$ and variants.

\begin{thm}
\label{HKR.3}
Assume that $S\to \Spec(R)$ is smooth.
For $i\in \Z$,
we have
\begin{equation}\label{eq:thm.HKR.3}
\Fil_{\HKR}^i\bHH(-/R),
\;
\Fil_{\rB}^i\bHC^-(-/R),
\;
\Fil_{\rB}^i\bHP(-/R)
\in
\logDA(S)_{\leq i}^\veff.
\end{equation}
In particular,
there are canonical morphisms 
\begin{align*}
&\tilde{\fil}_n \bHH(-/R) \simeq \Sigma^{2n,n}\tilde{\fil}_0  \bHH(-/R)  \to \Fil_n^{\HKR} \bHH(-/R),\\
&\tilde{\fil}_n \bHC^-(-/R) \simeq \Sigma^{2n,n}\tilde{\fil}_0  \bHC^-(-/R)  \to \Fil_n^{\rB} \bHC^-(-/R),\\
&\tilde{\fil}_n \bHP(-/R) \simeq \Sigma^{2n,n}\tilde{\fil}_0  \bHP(-/R)  \to \Fil_n^{\rB} \bHP(-/R),
\end{align*}
\end{thm}
\begin{proof}  We only need to show the first claim, namely, the statement of  \eqref{eq:thm.HKR.3}, since the remaining claims are a direct consequence of  Lemma \ref{slice.2}. 
By Proposition \ref{HKR.4},
we reduce to the case when $i=-1$.
Since $\Sigma^\infty X_+$ is compact in $\logDA(S)$,
it suffices to show the vanishing
\[
\Hom_{\logDA(S)}(\Sigma^\infty X_+,\gr^{\HKR}_{-j}\bHH(-/R))
\simeq
0
\]
for every integer $j>0$ and similar vanishings for $\gr^{\rB}_{-j}\bHC^-$ and $\gr^{\rB}_{-j}\bHP$.
By the smoothness assumption and Proposition \ref{HKR.5},
it suffices to show the vanishings
\[
H_{\Zar}^{-j}(X,\Omega_{X/R}^{-j})=0
\text{ and }
H_{\Zar}^{-2j}(X,\Omega_{X/R}^{\geq 0})=0
\]
for $j>0$.
This is clear.
\end{proof}

\begin{rmk}
Since $H_{Zar}^{2j}(X,\Omega_{X/R}^{\geq 0})$ does not vanish for $j\geq 0$ and nontrivial $X$, 
the proof of Theorem \ref{HKR.3} shows that $\Fil_\rB^i \bHC^-(-/R)$ and $\Fil_\rB^i \bHP(-/R)$ are not in $\logDA(S)_{\leq i}^\eff$.
On the other hand,
we have $\Fil_\HKR^i\bHH(-/R)\in \logDA_{\leq i}^\eff(S)$.
\end{rmk}

\begin{proof}[Proof of Theorem \ref{intro.2}]
(1) is a consequence of Lemma \ref{slice.2} and Theorems \ref{K.4} and \ref{HKR.3}.
(2) is a consequence of (1) and Proposition \ref{HKR.5},
and (3) is a consequence of (2).
\end{proof}
\begin{quest}Despite having a very different origin, 
Theorem \ref{HKR.3} exhibits a connection between the (very effective) slice filtration and the HKR filtration (and the Beilinson filtrations). We leave it as an open question whether the (very effective) slice and HKR filtrations agree on $\bHH(-/(R,P))$, and whether 
the slice and Beilinson filtrations agree on $\bHC^-(-/(R,P))$ and $\bHP(-/(R,P))$.
\end{quest}

\section{BMS filtration on logarithmic topological Hochschild homology}
Let us fix a prime $p$. 
We refer the reader to  \cite{BLPO2} for the definitions and the main properties of the category $\lQSyn$ of log quasi-syntomic rings,
the category $\lQRSPerfd$ of log quasi-regular semiperfectoid rings,
and the log quasi-syntomic topology on $\lQSyn$.
\begin{df} Let $\mathfrak{X}$ be a quasi-coherent
bounded $p$-adic formal log scheme. We say that $\mathfrak{X}$ is \emph{log quasi-syntomic} if strict \'etale locally it is isomorphic to $\Spf(A,M)^{a}$ for $(A,M)\in \lQSyn$.  
Let $\FlQSyn$ denote the category of quasi-compact quasi-separated log quasi-syntomic formal log schemes, as defined in \cite[Definition 3.1]{BLMP}: this is a log variant of the category of qcqs quasi-syntomic $p$-adic formal schemes of \cite{BL22} (i.e., qcqs $p$-adic formal schemes which are locally of the form $\Spf(R)$ for a $p$-adic quasi-syntomic ring $R$).
\end{df}

Our goal is to define the BMS filtrations on the motivic $\P^1$-spectra representing (log) $\THH$, $\TC^-$, $\TP$, and $\TC$, defined in \cite[Section 8]{BPO2} and explore their fundamental properties.

\begin{prop}
\label{BMS.4}
The presheaves $\THH(-;\Z_p)$, $\THH(-;\Z_p)_{hS^1}$, $\TC(-;\Z_p)^-$, and $\TP(-;\Z_p)$ on $\lQSyn^{op}$ are log quasi-syntomic sheaves.
\end{prop}
\begin{proof}
Argue as in \cite[Theorem 5.5]{zbMATH07729746}, but use the $p$-completed version of \cite[Theorem 2.3]{BLPO2} as \cite[Remark 4.9]{BMS19}.
See \cite[Proposition 3.10]{BLPO2} for the non $p$-completed version.
\end{proof}
In particular, for every $\kX \in \FlQSyn$, we can consider $\THH(\kX;\Z_p)$ and its cousins. 
\begin{rmk}
\label{BMS.19}

Given any $p$-complete ring $R$, we can consider the canonical $p$-completion functor $\lSm/R\to \mathrm{FlSm}/R$ from log smooth log schemes over $R$ to formal log smooth fs log schemes over $\Spf(R)$ as defined in \cite[Definition 2.2]{BLMP}.
Since for any  commutative ring $A$  we have an equivalence $\THH(A^\wedge_p;\Z_p) \simeq \THH(A;\Z_p)$, the functor $\THH(-;\Z_p)$ commutes with the $p$-completion functor. 
For this reason, we will implicitly extend the $\P^1$-spectrum $\bTHH$ in $\logSH(R)$ of \cite[\S 8]{BPO2} (denoted $\mathbf{logTHH}$ in loc.\ cit.) to a $\P^1$-spectrum, indicated with the same letters, in $\mathrm{logFSH}(R)$, where $\mathrm{logFSH}(R)$ is the category of log formal log motives,  constructed as in \cite[Definition 2.6]{BLMP}, using the strict Nisnevich topology and sheaves of spectra. For the same reason, we will implicitly consider the sheaf $\THH(-;\Z_p)$ and variants as defined on (log) schemes or on $p$-adic formal (log) schemes without explicitly mentioning it.
\end{rmk}

Let us quickly review how we can extend the BMS filtrations in \cite{BMS19} to pre-log rings,
see \cite{BLPO2}.
For $(A,M)\in \lQRSPerfd$,
the \emph{BMS filtrations on}
\[
\THH((A,M);\Z_p),
\;
\TC^-((A,M);\Z_p),
\;
\TP((A,M);\Z_p),
\;
\TC((A,M);\Z_p)
\]
are the double-speed Postnikov filtrations,
that is, $\Fil_\BMS^i:=\tau_{\leq 2i}$.
Equivalently, we have $\Fil^\BMS_i:=\tau_{\geq 2i-1}$.
By log quasi-syntomic descent,
we obtain the \emph{BMS filtrations} for $(A,M)\in \lQSyn$ too.
Observe that we have a natural equivalence of spectra
\begin{equation}
\label{BMS.0.1}
\gr^{\BMS}_i \TC((A,M);\Z_p)
\simeq
\fib(\gr^{\BMS}_i \TC^-((A,M);\Z_p)\xrightarrow{\varphi-\can} \gr^{\BMS}_i\TP((A,M);\Z_p)).
\end{equation}
For $\kX\in \FlQSyn$,
we also obtain the BMS filtrations on
\[
\THH(\kX;\Z_p),
\;
\TC^-(\kX;\Z_p),
\;
\TP(\kX;\Z_p),
\;
\TC(\kX;\Z_p)
\]
by Zariski descent. Note that the graded pieces satisfy quasi-syntomic descent as discussed in \cite[\S 3]{BLMP}.

For a $p$-adic formal scheme $\kX$ over $\Spf(R)$ for a quasi-syntomic ring $R$, 
let $\kX\times \P^n$ denote the $p$-adic formal scheme $\kX\times (\P^n_{\Spf(R)})_p^\wedge$ to simplify notation.

\begin{prop}
\label{BMS.14}
The presheaves 
\[\Fil^{\BMS}_i\THH(-;\Z_p),
\text{ }
\Fil^{\BMS}_i\TC^-(-;\Z_p),
\text{ and } \Fil^{\BMS}_i\TP(-;\Z_p)\] on $\FlQSyn$ are $(\P^n,\P^{n-1})$-invariant for $n\in \Z_{>0}$ and $i\in \Z\cup \{-\infty\}$.
\end{prop}
\begin{proof}
Using the completeness and exhaustiveness of the filtrations,
it suffices to show that the natural map of spectra
\[
\gr^\BMS_{i}\THH(\kX;\Z_p)
\to
\gr^\BMS_{i}\THH(\kX\times (\P^n,\P^{n-1});\Z_p)
\]
and the similar maps for $\TC^-$, $\TP$, and $\TC$ are equivalences.
This is a consequence of \cite[Theorem 1.3(3),(4)]{BLPO2} and \cite[Corollary 3.11]{BLMP}.
\end{proof}

\begin{prop}
\label{BMS.7}
For $(A,M)\in \lQSyn$ and integer $i<0$,
we have the vanishing
\[
\Fil_\BMS^i \TC((A,M);\Z_p)\simeq 0.
\]
\end{prop}
\begin{proof} By (log) quasi-syntomic descent, we reduce to the case  $(A,M)\in \lQRSPerfd$.
To conclude,
observe that we have $\pi_i\TC((A,M);\Z_p)\simeq 0$ for $i<-1$ using \cite[Theorem II.4.10]{zbMATH07009201} as in the proof of \cite[Proposition 7.16]{BMS19}.
\end{proof}

\begin{const}
\label{BMS.15}Let $A$ be a quasi-syntomic ring.
Apply the \'etale sheafification and $p$-completion functors to the first Chern class $c_1^\TC\colon R\Gamma_{\Zar}(A,\G_m)[1]\to \TC(A)$ to obtain the natural morphism
\[
\widehat{c}_1^\TC
\colon
R\Gamma_{\et}(A,\G_m)_p^\wedge[1]
\to
\TC(A;\Z_p).
\]
Apply $\pi_2(-)[2]$ to this to obtain the morphism
\[
T_p(A^\times)[2]
\to
(\pi_2\TC(A;\Z_p))[2],
\]
where $T_p(A^\times)$ is the $p$-adic Tate module of the abelian group $A^\times$. The \'etale sheafification of the left-hand side is $R\Gamma_{\et}(A,\G_m)_p^\wedge[1]$, and the quasi-syntomic sheafification of the right-hand side is $R\Gamma_{\syn}(A,\Z_p(1))[2]$ by \cite[Proposition 7.17]{BMS19}.
Hence
we obtain the induced natural morphism
\[
\widehat{c}_1^{\syn}
\colon
R\Gamma_{\et}(A,\G_m)_p^\wedge[1]
\to
R\Gamma_{\syn}(A,\Z_p(1))[2]
\]
for every quasi-syntomic ring $A$.
This agrees with $\widehat{c}_1^{\syn}$ in \cite[Theorem 7.5.6]{BL22} up to shift for quasi-regular semiperfectoid $\Z_p^\mathrm{cycl}=\Z_p[\mu_{p^\infty}]^\wedge_p$-algebra by comparing \cite[Theorem 6.7]{zbMATH07237241} and \cite[Notation 7.5.1]{BL22}, for quasi-regular semiperfectoid ring by quasi-syntomic descent for the cover $A\to A\widehat{\otimes}_{\Z_p} \Z_p^\mathrm{cycl}$,
and for quasi-syntomic ring by quasi-syntomic descent again.

For a quasi-regular semiperfectoid ring $A$ such that $A$ is $w$-local in the sense of \cite[Definition 1.6]{BSproet} (in particular any
Zariski cover of $\Spec{A}$ is split) and $A^\times$ is $p$-divisible (such rings form a base for the quasi-syntomic topology in light of the proof of \cite[Proposition 7.17]{BMS19}),
we see that $T_p(A^\times)[2]\simeq R\Gamma_{\et}(A,\G_m)_p^\wedge[1]$
is concentrated in degree $2$.
Hence for such a ring $A$,
we have a natural equivalence $\tau_{[1,2]}\widehat{c}_1^{\TC}\simeq \widehat{c}_1^{\syn}$.
In particular,
we obtain the natural commutative diagram
\[
\begin{tikzcd}
&
T_p(A^\times)[2]\ar[ld,"\widehat{c}_1^{\syn}"']\ar[d]\ar[rd,"\widehat{c}_1^\TC"]
\\
\tau_{[1,2]}\TC(A;\Z_p)\ar[r,leftarrow]&
\tau_{\geq 1}\TC(A;\Z_p)\ar[r]&
\TC(A;\Z_p).
\end{tikzcd}
\]
For $\kX\in \FlQSyn$,
we obtain the natural commutative diagram
\begin{equation}
\label{BMS.15.1}
\begin{tikzcd}
&
R\Gamma_{\et}(\kX,\G_m)_p^\wedge{[1]}\ar[ld,"\widehat{c}_1^{\syn}"']\ar[d]\ar[rd,"\widehat{c}_1^\TC"]
\\
R\Gamma_{\syn}(\kX,\Z_p(1)){[2]}\ar[r,leftarrow]&
\Fil^{\BMS}_1\TC(\kX;\Z_p)\ar[r]&
\TC(\kX;\Z_p)
\end{tikzcd}
\end{equation}
by quasi-syntomic descent.
\end{const}

For any $p$-adically complete ring $A$, write $A\langle x\rangle$ for the $p$-adic completion of the polynomial ring $A[x]$. 

\begin{lem}
\label{BMS.6}
For $(A,M)\in \lQRSPerfd$, the BMS filtrations on $\THH((A\langle x\rangle ,M);\Z_p)$ and $\THH((A\langle x,x^{-1}\rangle,M);\Z_p)$ are the double-speed Postnikov filtrations.
Similar results hold for $\TC^-$ and $\TP$ too.
\end{lem}
\begin{proof}
We focus on the case of $(A\langle x\rangle ,M)$ since the proofs are similar.
By definition of quasi-regular semiperfectoid ring, there exists a perfectoid ring $R$ with a map $R\to A$.
By base-change, we have equivalences of complexes
\[
\L_{(A[x],M)/R}
\simeq
\L_{R[x]/R}\otimes_R^\L A\oplus \L_{(A,M)/R}\otimes_R^\L R[x]
\simeq
A[x]\oplus  \L_{(A,M)/R}\otimes_R^\L R[x].
\]
After taking $p$-completions,
we get an equivalence of complexes
\[
\L_{(A\langle x\rangle,M)/R}
\simeq
A\langle x\rangle \oplus \L_{(A,M)/R}\widehat\otimes_R^\L R\langle x\rangle.
\]
Since $\L_{(A,M)/R}$ has $p$-complete Tor amplitude in degree $-1$ by e.g.\ \cite[Lemma 4.15]{BLPO2} and $\wedge_{A\langle x\rangle}^j A\langle x\rangle \simeq 0$ for $j\geq 2$,
$\wedge_{A\langle x\rangle}^i \L_{(A\langle x\rangle,M)/R}$ has $p$-complete Tor amplitude in $[-i,-i+1]$ for $i\geq 0$.
By \cite[Proposition 7.3]{BLPO2} for $\THH$ and \cite[Theorem 7.2(3),(4), Proposition 7.8]{BMS19} for $\TC^-$ and $\TP$,
we see that the $i$th graded pieces of $\THH((A\langle x\rangle,M);\Z_p)$, $\TC^-((A\langle x\rangle,M);\Z_p)$, and $\TP((A\langle x\rangle,M);\Z_p)$ are in $\rD^{[-2i,-2i+1]}(A\langle x\rangle)$.
Using this,
we deduce the claim.
\end{proof}

 The following proposition is crucial for us. 

\begin{prop}
\label{BMS.2}
For $\kX\in \FlQSyn_R$ and $i\in \Z\cup \{-\infty\}$,
the projective bundle formula for $\THH(\kX\times \P^1;
\Z_p)$ restricts to an equivalence
\[
\Fil^\BMS_i\THH(\kX\times \P^1;\Z_p)
\simeq
\Fil^\BMS_i\THH(\kX;\Z_p)
\oplus
\Fil^\BMS_{i-1}\THH(\kX;\Z_p).
\]
We have similar equivalences for $\TC^-$, $\TP$, and $\TC$ too.
\end{prop}
\begin{proof}
By log quasi-syntomic descent,
we reduce to the case when $\kX=\Spf(A,M)$ and $(A,M)\in \lQRSPerfd$.
By Remark \ref{BMS.19} and Lemma \ref{BMS.6},
the standard cover on $\P^1$ yields the cartesian square of $S^1$-equivariant filtered spectra
\[
\begin{tikzcd}
\THH((\P^1_{\Spf(A,M)})^\wedge_p;\Z_p)\ar[d]\ar[r]&
\THH((A\langle x\rangle,M);\Z_p)\ar[d]
\\
\THH((A\langle x^{-1}\rangle,M);\Z_p)\ar[r]&
\THH((A\langle x,x^{-1}\rangle,M);\Z_p)
\end{tikzcd}
\]
such that the filtrations on the three corners other than $\THH((\P^1_{\Spf(A,M)})^\wedge_p;\Z_p)$ are the double speed Postnikov filtrations. 
Next, we note that we have an $S^1$-equivariant equivalence
\begin{equation}
    \label{eq:bms_filtere}
\THH((A\langle x\rangle,M); \Z_p) \simeq \THH((A,M); \Z_p)\otimes_\Z \HH(\Z[x]) \end{equation}
and similarly for $A\langle x^{-1}\rangle$ and $A\langle x, x^{-1}\rangle$ (note that the right hand-side is already $p$-complete).
Hence,  
the descriptions of $\HH(\Z[x])$, $\HH(\Z[x^{-1}])$, and $\HH(\Z[x,x^{-1}])$ in Proposition \ref{BMS.16} tell that the double-speed Postnikov filtrations on these agree with the usual Postnikov filtrations: in other words, the equivalence \eqref{eq:bms_filtere} is an equivalence of $S^1$-equivariant filtered spectra.
Hence we obtain the $S^1$-equivariant equivalence of filtered spectra
\[
\THH((\P^1_{\Spf(A,M)})^\wedge_p;\Z_p)
\simeq
\THH(X;\Z_p)\otimes_\Z \HH(\P^1),
\]
where the filtration on $\HH(\P^1)$ is given by Proposition \ref{BMS.16}.
Hence we obtain the $S^1$-equivariant decomposition
\begin{equation}
\label{BMS.2.1}
\THH(\kX\times \P^1;\Z_p)
\simeq
\THH(\kX;\Z_p)
\oplus
(\THH(\kX;\Z_p)[1])\{-1\},
\end{equation}
where the filtration on $\THH(X;\Z_p)$ is the double-speed Postnikov filtration,
and
\[
\Fil_i(\THH(\kX;\Z_p)[1])\{-1\})
:=
\Fil_{i-1}\THH(\kX;\Z_p).
\]
From this,
we obtain the desired equivalence for $\THH$.
To obtain the desired equivalences for $\TC^-$ and $\TP$,
apply $(-)^{hS^1}$ and $(-)^{tS^1}$ to \eqref{BMS.2.1},
and use Lemma \ref{BMS.6}.
For $\TC$,
use \eqref{BMS.0.1}.
\end{proof}

\begin{const}
\label{BMS.18}
For $\kX\in \FlQSyn_R$,
Proposition \ref{BMS.2} yields the natural equivalences of spectra
\begin{gather*}
\gr^\BMS_i \THH(\kX;\Z_p)
\oplus
\gr^\BMS_{i-1}\THH(\kX;\Z_p)
\simeq
\gr^\BMS_i \THH(\kX\times \P^1;\Z_p),
\\
\gr^\BMS_i \TC^-(\kX;\Z_p)
\oplus
\gr^\BMS_{i-1}\TC^-(\kX;\Z_p)
\simeq
\gr^\BMS_i \TC^-(\kX\times \P^1;\Z_p),
\\
\gr^\BMS_i \TP(\kX;\Z_p)
\oplus
\gr^\BMS_{i-1}\TP(\kX;\Z_p)
\simeq
\gr^\BMS_i \TP(\kX\times \P^1;\Z_p),
\\
\gr^\BMS_i \TC(\kX;\Z_p)
\oplus
\gr^\BMS_{i-1}\TC(\kX;\Z_p)
\simeq
\gr^\BMS_i \TC(\kX\times \P^1;\Z_p).
\end{gather*}
As in Construction \ref{BMS.17}  these are identified with the projective bundle formulas for $\kX\times \P^1$
\begin{gather*}
\gr^\rN_i R\Gamma_{\cPrism}(\kX)\{i\}
\oplus
\gr^\rN_{i-1} R\Gamma_{\cPrism}(\kX)\{i-1\}[-2]
\xrightarrow{\simeq}
\gr^\rN_i R\Gamma_{\cPrism}(\kX\times \P^1)\{i\}
\\
\Fil^\rN_i R\Gamma_{\cPrism}(\kX)\{i\}
\oplus
\Fil^\rN_{i-1} R\Gamma_{\cPrism}(\kX)\{i-1\}[-2]
\xrightarrow{\simeq}
\Fil^\rN_i R\Gamma_{\cPrism}(\kX\times \P^1)\{i\}
\\
R\Gamma_{\cPrism}(\kX)\{i\}
\oplus
R\Gamma_{\cPrism}(\kX)\{i-1\}[-2]
\xrightarrow{\simeq}
R\Gamma_{\cPrism}(\kX\times \P^1)\{i\}
\\
R\Gamma_{\syn}(\kX,\Z_p(i))
\oplus
R\Gamma_{\syn}(\kX,\Z_p(i-1))[-2]
\xrightarrow{\simeq}
R\Gamma_{\syn}(\kX\times \P^1,\Z_p(i))
\end{gather*}
induced by $(1,\widehat{c}_1^{\syn} (\cO(1)))$.
These are special cases of \cite[Lemma 9.1.4(4)--(6)]{BL22} when $\kX$ has trivial log structure since the two constructions of $\widehat{c}_1^{\syn}$ discussed in Construction \ref{BMS.15} agree.
\end{const}

We now fix the spectrum $S$ of a quasi-syntomic ring $R$. 

\begin{df}
\label{BMS.1}
For $X\in \SmlSm/S$ and $i\in \Z\cup \{-\infty\}$,
Proposition \ref{BMS.2} yields a natural equivalence of spectra
\begin{equation}
\label{BMS.1.1}
\Fil^{\BMS}_i\THH(X;\Z_p)
\simeq
\Omega_{\P^1}\Fil^{\BMS}_{i+1}\THH(X;\Z_p).
\end{equation}
Here, we are implicitly composing with the $p$-completion functor. Note that every $X\in \SmlSm/S$ automatically satisfies the property that $X^\wedge_p \in \FlQSyn$. 
Using this as bonding maps,
Proposition \ref{BMS.14} enables us to define the complete exhaustive filtration 
\[
\Fil^{\BMS}_i\bTHH(-;\Z_p)
:=
(\Fil^{\BMS}_i \THH(-;\Z_p),\Fil^{\BMS}_{i+1} \THH(-;\Z_p),\ldots)
\]
on the $\P^1$-spectrum $\bTHH(-;\Z_p)\simeq \Fil^\BMS_{-\infty} \bTHH(-;\Z_p)\in \logSH(S)$.
We similarly define the complete exhaustive filtrations
\[
\Fil^{\BMS}_i\bTC^-(-;\Z_p),
\;
\Fil^{\BMS}_i\bTP(-;\Z_p),
\;
\Fil^{\BMS}_i\bTC(-;\Z_p).
\]
\end{df}

\begin{prop}
\label{BMS.5}
For $i\in \Z\cup \{-\infty\}$,
we have a natural equivalence in $\logSH(S)$
\[
\Fil^{\BMS}_i\bTHH(-;\Z_p)(1)[2]
\simeq
\Fil^{\BMS}_{i+1}\bTHH(-;\Z_p).
\]
We have similar equivalences for $\bTC^-$, $\bTP$, and $\bTC$ too.
\end{prop}
\begin{proof}
For $\bTHH$,
this is a direct consequence of \eqref{BMS.1.1}.
The proofs of the other cases are similar. 
\end{proof}
We recall now from \cite[Theorem 1.3 and \S 7.4]{BLPO2} (see also \cite[\S 3]{BLMP}) that the Nygaard-completed absolute log prismatic cohomology 
$\cPrism:= R\Gamma_{\rm lqsyn}(-,\pi_{0}\TP(-)_p^\wedge)$  equipped with the Nygaard filtration define functors
\begin{align*}
&\cPrism\in \Fun(\lQSyn,\cD(\Z_p)), &\Fil^{\rN}_{i}\cPrism \in \Fun(\lQSyn, \widehat{\cD\cF}(\Z_p))
\end{align*}
analogously to \cite[\S 7]{BMS19}, where  $\Fil^{\rN}_{i}\cPrism$ denotes  the filtered piece $\geq i$ with respect to the Nygaard filtration. By \cite[Theorem 2.9]{BLPO2}, these functors are in fact quasi-syntomic sheaves.  
Define the Breuil--Kisin twists 
\begin{align*}
    \cPrism\{1\}&:= R\Gamma_{\rm lqsyn}(-,\pi_{2}\TP(-)_p^\wedge)[-2]\\
    \cPrism\{-1\}&:= R\Gamma_{\rm lqsyn}(-,\pi_{-2}\TP(-)_p^\wedge)[2]
\end{align*}
with the Nygaard filtrations given by unfolding the double-speed Postnikov filtration.  
As in \cite[Theorem 1.12 (3)]{BMS19}, we get the twisted prismatic cohomology $\cPrism\{i\}$ by taking tensor powers in $\Fun(\lQSyn,\widehat{\cD\cF}(\Z_p))$, with the induced filtration. In light of \cite[Theorems 3.15, 3.16]{BLMP} we can make the following definition: 
\begin{df} 
For integers $d$ and $i$,
we define $\P^1$-spectra
\begin{gather*}
\Fil^\rN_i\bPrism\{d\}
:=
(
\Fil^\rN_i\cPrism\{d\},
\Fil^\rN_{i+1}\cPrism\{d+1\}[2],
\ldots
),
\\
\bPrism\{d\}
:=
(
\cPrism\{d\},
\cPrism\{d+1\}[2],
\ldots
),
\\
\bM \Z_p^{\syn}(d)
:=
\fib(\varphi_p-\can\colon \Fil^\rN_d\bPrism\{d\}\to \bPrism\{d\}),
\end{gather*}
where the bonding maps for $\Fil^\rN_n\bPrism\{d\}$, $\bPrism\{d\}$, and $\bM\Z_p^{\syn}(d)$ are obtained by the projective bundle formulas in Construction \ref{BMS.18},
and the morphisms $\varphi_p$ and $\can$ are the pointwise cyclotomic trace and canonical morphisms.
We will omit $\{d\}$ and $(d)$ if $d=0$.
\end{df}
Note that in \cite[Theorems 3.15, 3.16]{BLMP}, (formal variants of) $\bPrism$, $\Fil^\rN_0\bPrism$, and $\bM\Z_p^{\syn}$ were denoted by $\bE^{\cPrism}$, $\bE^{\Fil \,\cPrism}$, and $\bE^{\mathrm{Fsyn}}$.

\begin{prop}
\label{BMS.9}
For $i\in \Z$,
we have natural equivalences
\begin{gather*}
\gr^{\BMS}_i \bTHH(-;\Z_p)
\simeq
\gr^\rN_i \bPrism\{i\}[2i],
\\
\gr^{\BMS}_i \bTC^-(-;\Z_p)
\simeq
\Fil^\rN_i\bPrism\{i\}[2i],
\\
\gr^{\BMS}_i \bTP(-;\Z_p)
\simeq
\bPrism\{i\}[2i],
\\
\gr^{\BMS}_i \bTC(-;\Z_p)
\simeq
\bM \Z_p^{\syn}(i)[2i]
\end{gather*}
in $\logSH(S)$.
\end{prop}
\begin{proof}
By \cite[Theorem 1.3(3),(4)]{BLPO2},
it remains to show that the bonding maps for the left-hand sides can be identified with the bonding maps for the right-hand sides.
This is a consequence of Construction \ref{BMS.18}.
\end{proof}
\begin{rmk}
    The right-hand side of the displayed equivalences in Proposition \ref{BMS.9} are naturally constructed as objects of $\logDA(S, \Z_p)$, while the filtered spectra  $\Fil^{\BMS}_{i+1}\bTHH(-;\Z_p)$ and variants are in $\logSH(S, \Sph_p)$, where $\Sph_p$ denotes the $p$-completed sphere spectrum. Note that, in particular, we obtain that each graded piece $\gr^{\BMS}_i $ is in $\logDA(S, \Z_p)$.
\end{rmk}

The Tate twist is related to the Breuil-Kisin twist in the following sense:

\begin{prop}
\label{BMS.10}
For $d\in \Z$ and $i\in \Z\cup \{-\infty\}$,
we have a natural equivalence
\[
\Fil^\rN_i\bPrism(d)
\simeq
\Fil^\rN_{i+d}\bPrism\{d\}
\]
in $\logSH(S)$.
\end{prop}
\begin{proof}
We have
\[
\Fil^\rN_i\bPrism
:=
(\Fil^\rN_i\cPrism,
\Fil^\rN_{i+1}\cPrism\{1\}[2],
\ldots
).
\]
Tensoring with $(d)[2d]$ in $\logSH(S)$ means shifting $d$ terms,
so we have
\[
\Fil^\rN_i\bPrism(d)[2d]
\simeq
(\Fil^\rN_{i+d}\cPrism\{d\}[2d],
\Fil^\rN_{i+d+1}\cPrism\{d+1\}[2d+1],
\ldots
).
\]
We obtain the desired equivalence after applying $[-2d]$.
\end{proof}

If $i=-\infty$,
then observe that Proposition \ref{BMS.10} yields an equivalence
\[
\bPrism(d)\simeq \bPrism\{d\}
\]
in $\logSH(S)$ for $d\in \Z$.

The notation $\bM \Z_p^{\syn}(i)$ is compatible with the Tate twist in the following sense:

\begin{prop}
\label{BMS.11}
For $d,i\in \Z$, we have a natural equivalence
\[
(\bM \Z_p^{\syn}(i))(d)
\simeq
\bM \Z_p^{\syn}(i+d)
\]
in $\logSH(S)$.
\end{prop}
\begin{proof}
This is a direct consequence of Proposition \ref{BMS.10}.
\end{proof}

\begin{thm}
\label{BMS.3}
For $i\in \Z$,
\[
\Fil_\BMS^i\bTHH(-;\Z_p),
\;
\Fil_\BMS^i\bTC^-(-;\Z_p),
\;
\Fil_\BMS^i\bTP(-;\Z_p),
\;
\Fil_\BMS^i\bTC(-;\Z_p)
\]
are in $\logSH(S)_{\leq i}^{\veff}$. In particular, there are canonical morphisms
\[
\tilde{\fil}_n \THH(-;\Z_p) \simeq \Sigma^{2n,n} \tilde{\fil}_0 \THH(-;\Z_p) \to \Fil_n^{\BMS} \THH(-;\Z_p), 
\]
and similarly for $\bTC^{-}$, $\bTP$ and $\bTC$. 
\end{thm}
\begin{proof}
By Proposition \ref{BMS.5},
we reduce to the case when $i=-1$.
It suffices to show the vanishing
\[
\Hom_{\logSH(S)}(\Sigma^\infty X_+,\Fil_\BMS^{-1}\bTHH(-;\Z_p))
=
0
\]
and similar vanishings for $\bTC^-$, $\bTP$, and $\bTC$.
Since $\Fil_\BMS^{-1}\THH(-;\Z_p)=0$ by construction and $\Fil_\BMS^{-1}\TC(-;\Z_p)=0$ by Proposition \ref{BMS.7},
the claim holds for $\bTHH$ and $\bTC$.
Using the compactness of $\Sigma^\infty X_+$ in $\logSH(S)$,
it suffices to show the vanishing
\[
\Hom_{\logSH(S)}(\Sigma^\infty X_+,\gr^{\BMS}_{-j}\bTC^-(-;\Z_p))
=
0
\]
for every integer $j>0$ and similar vanishing for $\bTP$.
By Proposition \ref{BMS.9},
it suffices to show the vanishings
\[
H_{\sNis}^{-2j}(X,\cPrism\{-j\})=0.
\]
This follows from $\cPrism_{(A,M)}\{-j\}\in \rD^{\geq 0}(A)$ for $(A,M)\in \lQSyn$,
which can be deduced from the case of $(A,M)\in \lQRSPerfd$ by log quasi-syntomic descent.
\end{proof}

\begin{quest} As for the HKR filtration, we leave it as an open question whether the slice and BMS filtrations agree on $\bTHH(-;\Z_p)$, $\bTC^-(-;\Z_p)$, $\bTP(-;\Z_p)$, and $\bTC(-;\Z_p)$. Since $T$ is the analog of $S^2$, the double-speed convergence of the very effective slice filtration on $\KGL$ gives at least an analogy with the double-speed Postnikov filtration defining the BMS filtration.
\end{quest}

\section{The characteristic \texorpdfstring{$p$}{p} case}

Throughout this section,
$k$ is a perfect field of characteristic $p$.
Our goal is to identify the graded pieces of the filtered spectra $\bTHH(-;\Z_p)$, $\bTC^-(-;\Z_p)$, $\bTP(-;\Z_p)$, and $\bTC(-;\Z_p)$ in this case.
We will use \cite{BMS19} for the identifications at the level of $S^1$-spectra,
and then we will use the projective bundle formulas in \cite{BL22} to identify the bonding maps.

\begin{const}
\label{p.7}
For a log smooth saturated pre-log $k$-algebra $(A,M)$,
recall from \cite[\S 4.3]{BLMP} that the chain complex $W\Omega_{(A,M)/k}$ comes equipped with the Nygaard filtration $\Fil_\rN^i W\Omega_{(A,M)/k}$ given by the subcomplex
\begin{align*}
p^{i-1}VW(A)
\to &
p^{i-2}VW\Omega_{(A,M)/k}^1
\to
\cdots
\to
pVW\Omega_{(A,M)/k}^{i-2}
\\
\to
&
VW\Omega_{(A,M)/k}^{i-1}
\to
W\Omega_{(A,M)/k}^i
\to
W\Omega_{(A,M)/k}^{i+1}
\to
\cdots.
\end{align*}
For $X\in \lSm/k$,
the log Crystalline cohomology $R\Gamma_{\crys}(X):=R\Gamma_\Zar(X,W\Omega_{X/k})$ is then equipped with the Nygaard filtration $\Fil^\rN_i R\Gamma_{\crys}$ by Zariski descent.
Furthermore,
we obtain a natural equivalence of complexes
\begin{equation}
\label{p.7.1}
\gr^\rN_i R\Gamma_{\crys}(X)
\simeq
R\Gamma_\Zar(X,\tau^{\leq i} \Omega_{X/k})
\end{equation}
from \cite[(4.22.1), (4.22.2)]{BLMP}.

By \cite[Theorem 4.30]{BLMP},
we have a natural equivalence of filtered $\E_\infty$-rings
\[
R\Gamma_{\cPrism}(X)
\simeq
R\Gamma_{\crys}(X).
\]
\end{const}

\begin{const}
\label{p.12}
For $X\in \lSm/k$,
recall from \cite[Construction 7.3.1]{BL22} the morphism of complexes
\[
c_1^{\crys}
\colon
R\Gamma_{\et}(\ul{X},\G_m)[1]
\to
\Fil^\rN_1R\Gamma_{\crys}(\ul{X})[2].
\]
Compose this with the canonical morphism $\Fil^\rN_1R\Gamma_{\crys}(\ul{X})[2]\to \Fil^\rN_1R\Gamma_{\crys}(X)[2]$ to obtain
\[
c_1^{\crys}
\colon
R\Gamma_{\et}(\ul{X},\G_m)[1]
\to
\Fil^\rN_1R\Gamma_{\crys}(X)[2].
\]
As a consequence of \cite[Proposition 7.5.5]{BL22},
we see that the triangle
\begin{equation}
\label{p.12.1}
\begin{tikzcd}
&
R\Gamma_{\et}(\ul{X},\G_m)[1]\ar[ld,"c_1^{\cPrism}"']\ar[rd,"c_1^\crys"]
\\
\Fil^\rN_1R\Gamma_{\cPrism}(X)[2]\ar[rr,"\simeq"]&
&
\Fil^\rN_1R\Gamma_{\crys}(X)[2]
\end{tikzcd}
\end{equation}
commutes.
\end{const}

\begin{const}
For $X\in \lSm/k$,
recall from \cite[Theorem 4.12]{MR1463703} that we have the inverse Cartier operator,
which is a natural isomorphism of graded $\cO_X$-algebras
\[
C_{X/k}^{-1}
\colon
\Omega_{X/k}^*
\xrightarrow{\cong}
\cH^*(\Omega_{X/k})
\]
sending $d\log x$ to $d\log x$ for every local section $x$ of $\cM_X$.
Using this,
we have the composite morphism
\[
c_1^{\Hod}
\colon
R\Gamma_{\Zar}(\ul{X},\G_m)[1]
\xrightarrow{c_1^{\Hod}}
R\Gamma_{\Zar}(X,\Omega_{X/k}^1)[1]
\xrightarrow{C_{X/k}^{-1}}
R\Gamma_{\Zar}(X,\cH^1(\Omega_{X/k}))[1],
\]
and we have the projective bundle formulas for $\P^1$
\begin{gather*}
R\Gamma_{\Zar}(X,\cH^i(\Omega_{X/k}))
\oplus
R\Gamma_{\Zar}(X,\cH^{i-1}(\Omega_{X/k}))[-1]
\xrightarrow{\simeq}
R\Gamma_{\Zar}(X\times \P^1,\cH^i(\Omega_{X/k})),
\\
R\Gamma_{\Zar}(X,\tau^{\leq i}\Omega_{X/k})
\oplus
R\Gamma_{\Zar}(X,\tau^{\leq i-1}\Omega_{X/k})[-1]
\xrightarrow{\simeq}
R\Gamma_{\Zar}(X\times \P^1,\tau^{\leq i}\Omega_{X/k})
\end{gather*}
induced by $(1,c_1^{\Hod}(\cO(1)))$.
Using \eqref{p.7.1},
the last equivalence can be identified with
\[
\gr^\rN_iR\Gamma_{\crys}(X)\oplus \gr^\rN_{i-1}R\Gamma_{\crys}(X)[-1] 
\xrightarrow{\simeq}
\gr^\rN_iR\Gamma_{\crys}(X\times \P^1).
\]
\end{const}

\begin{prop}
\label{p.10}
For $X\in \lSm/k$ and integer $i$,
the natural square
\[
\begin{tikzcd}[column sep=huge]
\Fil^\rN_iR\Gamma_{\crys}(X)\oplus \Fil^\rN_{i-1}R\Gamma_{\crys}(X)[-1] \ar[r,"{(1,c_1^\crys(\cO(1))}"]\ar[d]&
\Fil^\rN_iR\Gamma_{\crys}(X\times \P^1)\ar[d]
\\
\gr^\rN_iR\Gamma_{\crys}(X)\oplus \gr^\rN_{i-1}R\Gamma_{\crys}(X)[-1] \ar[r,"{(1,c_1^\Hod(\cO(1))}"]&
\gr^\rN_iR\Gamma_{\crys}(X\times \P^1)
\end{tikzcd}
\]
commutes.
\end{prop}
\begin{proof}
It suffices to show that the image of $c_1^{\crys}(\cO(1))\in H^2(\Fil^\rN_1R\Gamma_{\crys}(X\times \P^1))$ in $H^2(\gr^\rN_1R\Gamma_{\crys}(X\times \P^1))$ agrees with $c_1^{\Hod}(\cO(1))$ since these induce the projective bundle formulas for $\P^1$.
For this,
we reduce to the case when $X=\Spec(k)$ since the general case can be shown by considering the pullbacks of the first Chern classes.

We set $A:=k[x,x^{-1}]$ for simplicity of notation.
Since $\cO(1)\in R\Gamma_{\Zar}(\P_k^1,\G_m)$ is the image of $x^{-1}\in A^*$ under the boundary map $A^*\cong R\Gamma_{\Zar}(\G_{m,k},\G_m) \to R\Gamma_{\Zar}(\P_k^1,\G_m)$,
it suffices to show that the image of $c_1^{\crys}(x^{-1})\in H^1(\Fil^\rN_1W\Omega_{A/k})$ in $H^1(\gr^\rN_1 W\Omega_{A/k})$ agrees with $c_1^{\Hod}(x^{-1})$.

Consider the canonical morphism $W\Omega_{A/k}\to \Omega_{A/k}$,
which induces $\Fil^\rN_1W\Omega_{A/k}\to \Fil^\Hod_1\Omega_{A/k}$ and hence $\gamma^{\dR}\colon \gr^\rN_1 W\Omega_{A/k}\to \gr^\Hod_1 \Omega_{A/k}$.
The equivalence $\gr^\rN_1 W\Omega_{A/k}\simeq \tau^{\leq 1}\Omega_{A/k}$ is explicitly given by the isomorphism of chain complexes
\[
\begin{tikzcd}
VW(A)/pVW(A)\ar[d,"d"']\ar[r,"F/p"]&
A\ar[d,"d"]
\\
W\Omega_{A/k}^1/VW\Omega_{A/k}^1\ar[r,"F"]&
Z\Omega^1_{A/k},
\end{tikzcd}
\]
where the lower horizontal map is due to \cite[(I.3.11.3)]{MR565469}.
Together with \cite[Proposition I.3.3]{MR565469},
we see that $\gamma^{dR}$ is given by the composite morphism $\tau^{\leq 1}\Omega_{A/k} \to H^1(\Omega_{A/k})[-1] \xrightarrow{C_{X/k}} \Omega_{A/k}^1$.
In particular, $H^1(\gamma^{dR})$ is an isomorphism.
Hence to compare $c_1^{\crys}(x^{-1})$ and $c_1^\Hod(x^{-1})$ in $H^1(\gr^\rN_1W\Omega_{A/k})$,
it suffices to compare their images in $H^1(\gr^\Hod_1 \Omega_{A/k})$.

By \cite[Example 5.2.4, Theorem 7.6.2]{BL22},
the image of $c_1^{\crys}(x^{-1})$ in $H^1(\gr^\Hod_1 \Omega_{A/k})$ is identified with $c_1^\Hod(x^{-1})=d\log x\in \Omega_{A/k}^1$.
To conclude,
observe that $H^1(\gamma^{dR})=C_{X/k}\colon H^1(\Omega_{A/k})\to \Omega_{A/k}$ sends $c_1^\Hod(x^{-1})=d\log x$ to $d\log x$.
\end{proof}

\begin{df}
\label{p.1}
The \emph{strict pro-\'etale topology}
on the category of log schemes is the topology generated by the families $\{X_i\to X\}_{i\in I}$ such that $\{\ul{X_i}\to \ul{X}\}_{i\in I}$ is a pro-\'etale covering
Let $\sproet$
be the shorthand for this topology. We denote by $L_{\sproet}(-)$ the strict pro-\'etale sheafification functor.
\end{df}
Recall the following construction from, e.g., \cite{LORENZON2002247}. 
\begin{df} Let $X$ be a fine log scheme over $k$ and let $\mathrm{dlog}\colon \mathcal{M}_X^{\rm gp}\to \Omega^1_{X/k}$ be the natural map. For every integer $r\geq 1$, write $\mathrm{dlog}[\cdot]\colon (\mathcal{M}_X^{\rm gp})^{\otimes i} \to W_r \Omega^{i}_{X/k}$ for the map of strict \'etale sheaves locally given by $m_1\otimes\ldots \otimes m_i \to \mathrm{dlog}[m_1] \wedge \ldots \wedge \mathrm{dlog}[m_i]$, where $[m]$ denotes the lift of $m$ as in \cite[2.5]{LORENZON2002247}. We denote by $W_r\Omega^i_{X/k, \log} \subseteq W_r\Omega^i_{X/k}$ the strict \'etale subsheaf generated by the image of $\mathrm{dlog}[\cdot]$, and by $W\Omega^i_{X/k, \log}$ the limit $\lim_r W_r\Omega^i_{X/k, \log}$ as a strict pro-\'etale sheaf.
\end{df}
\begin{prop}
\label{p.2}
For $X\in \SmlSm/k$ and integer $i$,
the sequence of chain complexes of strict pro-\'etale sheaves
\[
0
\to
W\Omega_{X/k,\log}^i[-i]
\to
\Fil^\rN_i W\Omega_{X/k}^\bullet
\xrightarrow{\varphi/p^i-1}
W\Omega_{X/k}^\bullet
\to
0
\]
is exact in each degree.
Furthermore,
we have a natural equivalence of complexes of strict pro-\'etale sheaves $W\Omega_{X/k,\log}^i\simeq R\lim_m W_m\Omega_{X/k,\log}^i$ (that is, $W\Omega_{X/k,\log}^i$ is also the derived inverse limit).
\end{prop}
\begin{proof}
This is formally identical to \cite[Proposition 8.4]{BMS19},
 using \cite[Corollary 2.14]{LORENZON2002247} instead of \cite[Theorem I.5.7.2]{MR565469}.
\end{proof}
\begin{rmk}\label{rmk:let=ket} Recall the following fact. For any fs log scheme $X$ and any quasi-coherent sheaf $\cF$ on $\underline{X}$, we have $R\Gamma_{\setale}(X, \cF) \simeq R\Gamma_{\ket}(X,\cF)$ (this is due to Kato, see e.g.\ \cite[Proposition 9.2.3]{BPO}), and this agrees with $R\Gamma_{\Zar}(X, \cF)$ and $R\Gamma_{\sNis}(X, \cF)$ by \cite[(9.1.1)]{BPO}. Moreover, we have an equivalence 
\[ R\Gamma(X, L_\letale \cF) \simeq \colim_{Y\in X_{\rm div}} R\Gamma_{\ket}(Y, \cF) \]
by \cite[Theorem 5.1.2]{BPO}. Hence if $\cG$ is a complex of bounded below $(\P^n, \P^{n-1})$-invariant coherent sheaves on $\SmlSm/k$, then $\cG$ is a complex of log \'etale hypersheaves.
\end{rmk}

For $X\in \lSm/k$ and integer $i\geq 0$,
we set $B\Omega_{X/k}^i:=\im(\Omega_{X/k}^{i-1}\xrightarrow{d} \Omega_{X/k}^i)$ and $Z\Omega_{X/k}^i:=\ker(\Omega_{X/k}^i\xrightarrow{d}\Omega_{X/k}^{i+1})$. The following result is a consequence of \cite[Theorem 4.2]{Merici} and Remark \ref{rmk:let=ket} above, using, for example, the fact that the above presheaves restricted to smooth $k$-schemes are reciprocity sheaves (see \cite[\S 11]{BRS}).
We also provide another direct proof.

\begin{prop}
\label{p.3}
The presheaves of complexes
\[
\Omega^i,
\text{ }
B\Omega^i,
\text{ }
Z\Omega^i,
\text{ }
\tau^{\leq i}\Omega,
\text{ }
W_m\Omega,
\text{ }
W\Omega
\]
on $\SmlSm/k$ given by
\begin{align*}
X\mapsto &
R\Gamma_\Zar(X,\Omega_{X/k}^i),
\text{ }
R\Gamma_\Zar(X,B\Omega_{X/k}^i),
\text{ }
R\Gamma_\Zar(X,Z\Omega_{X/k}^i),
\\
&
R\Gamma_\Zar(X,\tau^{\leq i}\Omega_{X/k}),
\text{ }
R\Gamma_\Zar(X,W_m\Omega_{X/k}),
\text{ }
R\Gamma_\Zar(X,W\Omega_{X/k})
\end{align*}
are $(\P^n,\P^{n-1})$-invariant and log \'etale hypersheaves for all integers $m,n>0$ and $i$.
\end{prop}
\begin{proof}
By \cite[Corollary 9.2.2, Proposition 9.2.6]{BPO},
$\Omega^i$ is $(\mathbb{P}^n,\mathbb{P}^{n-1})$-invariant and a log \'etale hypersheaf.
We proceed by induction to show that $B\Omega^i$ is $(\mathbb{P}^n,\mathbb{P}^{n-1})$-invariant and a log \'etale hypersheaf.
The claim is clear if $i=0$ since $B\Omega_{X/k}^0=0$ for $X\in \mathrm{SmlSm}/k$.
If the claim holds for $i$,
then use the obvious exact sequence
\[
0
\to
B\Omega_{X/k}^i
\to
\Omega_{X/k}^i
\to
Z\Omega_{X/k}^{i+1}
\to
0
\]
and the exact sequence
\[
0
\to
B\Omega_{X/k}^{i+1}
\to
Z\Omega_{X/k}^{i+1}
\to
\Omega_{X/k}^{i+1}
\to
0
\]
obtained by \cite[(2.12.2)]{AST_1994__223__221_0} to show the claim for $i+1$.
We also deduce that $Z\Omega^i$ is $(\mathbb{P}^n,\mathbb{P}^{n-1})$-invariant and a log \'etale hypersheaf.
It follows that $\tau^{\leq i}\Omega$ is $(\mathbb{P}^n,\mathbb{P}^{n-1})$-invariant and a log \'etale hypersheaf.

The claim for $\Omega^i$ for all $i$ implies the claim for $\Omega$.
By induction on $m$,
we deduce the claim for $W_m\Omega$.
After taking limits over $m$,
we deduce the claim for $W\Omega$.
\end{proof}

We note that $R\Gamma_{Zar}(X,W\Omega_{X/k})$ represents Hyodo-Kato cohomology with the trivial log structure on $k$,
see \cite[Corollary 1.23]{LORENZON2002247}.

\begin{prop}
\label{p.11}
The presheaf of complexes
\[
L_{\seta} W_m\Omega_{\log}^i,
\text{ }
L_{\sproet}W\Omega_{\log}^i
\]
on $\SmlSm/k$ given by
\[
X
\mapsto
R\Gamma_{\seta}(X,W_m\Omega_{X/k,\log}^i),
\text{ }
R\Gamma_{\sproet}(X,W\Omega_{X/k,\log}^i)
\]
are $(\P^n,\P^{n-1})$-invariant and log \'etale hypersheaves for all integers $m,n>0$ and $i$.
\end{prop}
\begin{proof}
Consider the exact sequence of strict \'etale sheaves
\[
0
\to
\Omega_{X/k,\log}^i
\to
\Omega_{X/k}^i
\to
\Omega_{X/k}^i/B\Omega_{X/k}^i
\to
0
\]
obtained by \cite[Proposition 2.13]{LORENZON2002247}.
By \cite[(2.12)]{LORENZON2002247}, Remark \ref{rmk:let=ket} and induction on $m$,
we see that $L_{\seta} W_m\Omega_{\log}^i$ is $(\P^n,\P^{n-1})$-invariant and a log \'etale hypersheaf.
Together with Proposition \ref{p.2},
we deduce that $L_{\sproet}W\Omega_{\log}^i$ is $(\P^n,\P^{n-1})$-invariant and a log \'etale hypersheaf. 
\end{proof}

\begin{const}
\label{p.6}
For $X\in \lSm/k$ and integer $i$,
we have a natural equivalence of complexes
\begin{equation}
\label{p.6.1}
R\Gamma_{\syn}(X,\Z_p(i))
\simeq
R\Gamma_{\sproet}(X,W\Omega_{X/k,\log}^i)
\end{equation}
by \cite[Theorem 4.30]{BLMP} and Proposition \ref{p.2}, and see \cite[Theorem 4.27]{BLMP} for the $\varphi$-compatibility. 
Together with \cite[Theorem 7.3.5]{BL22},
we see that $c_1^{\crys}\colon R\Gamma_{\et}(X,\G_m)[1]\to R\Gamma_{\crys}(X)[2]$ naturally factors through a morphism of complexes
\[
c_1^{W\Omega_{\log}}
\colon
R\Gamma_{\et}(X,\G_m)[1]
\to
R\Gamma_{\sproet}(X,W\Omega_{X/k,\log}^1)[1].
\]
As a consequence of \cite[Proposition 7.5.5]{BL22},
we see that the triangle
\begin{equation}
\label{p.6.2}
\begin{tikzcd}
&
R\Gamma_{\et}(\ul{X},\G_m)[1]\ar[ld,"c_1^{\syn}"']\ar[rd,"c_1^{W\Omega_{\log}}"]
\\
R\Gamma_{\syn}(X)[2]\ar[rr,"\simeq"]&
&
R\Gamma_{\sproet}(X,W\Omega_{X/k,\log}^1)[2]
\end{tikzcd}
\end{equation}
commutes.
\end{const}

\begin{const}
For $X\in \lSm/k$,
recall from \cite[Construction 7.3.1]{BL22} the morphism of complexes
\[
c_1^{\syn}
\colon
R\Gamma_{\et}(\ul{X},\G_m)[1]
\to
\Fil^\rN_1R\Gamma_{\crys}(\ul{X})[2].
\]
Compose this with the canonical morphism $\Fil^\rN_1R\Gamma_{\crys}(\ul{X})[2]\to \Fil^\rN_1R\Gamma_{\crys}(X)[2]$ to obtain
\[
c_1^{\crys}
\colon
R\Gamma_{\et}(\ul{X},\G_m)[1]
\to
\Fil^\rN_1R\Gamma_{\crys}(X)[2].
\]
\end{const}

\begin{df}
\label{p.4}
By Proposition \ref{p.3},
we can define the objects
\begin{gather*}
\Fil^\conj_i\bOmega
:=
(\tau_{\geq i}\Omega,(\tau_{\geq i+1}\Omega)[1],\ldots),
\\
\Fil^\rN_i\bW\bOmega
:=
(\Fil^\rN_iW\Omega,
(\Fil^\rN_{i+1}W\Omega)[1],
\ldots),
\\
L_{\sproet}\bW \bOmega_{\log}^i
:=
(L_{\sproet}W\Omega_{\log}^i,
L_{\sproet}W\Omega_{\log}^{i+1}[1],
\ldots)
\end{gather*}
in $\logSH(k)$,
where the bonding maps are given by the following composite morphisms:
\begin{gather*}
\tau_{\geq i} \Omega
\xrightarrow{c_1^\Hod}
\tau_{\geq i} (\Omega_{\P^1}\Omega[1])
\xrightarrow{\simeq}
\Omega_{\P^1}(\tau_{\geq i+1} \Omega)[1],
\\
\Fil^\rN_iW\Omega
\xrightarrow{c_1^\crys}
\Fil^\rN_i(\Omega_{\P^1}W\Omega[1])
\xrightarrow{\simeq}
\Omega_{\P^1}(\Fil^\rN_{i+1}W\Omega)[1],
\\
L_{\sproet}W\Omega_{\log}^i
\xrightarrow{c_1^{W\Omega_{\log}}}
L_{\sproet}\Omega_{\P^1}W\Omega_{\log}^{i+1}[1]
\xrightarrow{\simeq}
\Omega_{\P^1}L_{\sproet}W\Omega_{\log}^{i+1}[1].
\end{gather*}
\end{df}

\begin{prop}
\label{p.5}
For $i\in \Z$,
we have natural equivalences
\begin{gather*}
\gr^\rN_i \bPrism
\simeq
\Fil^\conj_i \bOmega,
\\
\Fil^\rN_i \bPrism
\simeq
\Fil^\rN_i\bW\bOmega,
\\
\bPrism
\simeq
\bW\bOmega,
\\
\bM \Z_p^{\syn}(i)
\simeq
L_{\sproet}\bW\bOmega_{\log}^i[-i]
\end{gather*}
in $\logSH(k)$.
\end{prop}
\begin{proof}
By \cite[Theorem 4.30, Proposition 5.1]{BLMP} and \eqref{p.6.1},
it suffices to show that the bonding maps on both sides are compatible.
This follows from Proposition \ref{p.10} and the commutativity of \eqref{p.12.1} and \eqref{p.6.2}.
\end{proof}

\begin{proof}[Proof of Theorem \ref{intro.1}]
(1) is a consequence of Lemma \ref{slice.2} and Theorems \ref{K.4} and \ref{BMS.3}.
(2) is a consequence of (1) and Propositions \ref{BMS.9} and \ref{p.5},
and (3) is a consequence of (2).
\end{proof}

\section{Very effective slices of Kummer \'etale K-theory}\label{sec:Kummer}

Let $k$ be a perfect field.
In this section,
we construct a morphism in $\logSH_\ket^{\wedge}(k)$ (the hypercomplete version of $\logSH_\ket(k)$) representing a morphism from the Lichtenbaum \'etale cohomology to the syntomic cohomology induced by the cyclotomic trace.
We also discuss the very effective slices of Kummer \'etale $K$-theory.

By Proposition \ref{p.11}, the functors
$L_{\sproet}W\Omega_{\log}^i$ are  log \'etale hypersheaves,
and represent $p$-adic log syntomic cohomology.
Hence we obtain the morphism $L_{\ket}\MZ\to \MZ_p^\syn$ in $\logSH_{\ket}^{\wedge}(k)$ by adjunction from the morphism $\MZ\to \MZ_p^\syn$ in $\logSH(k)$ induced by the motivic cyclotomic trace map.
We will show that $L_{\ket}\MZ$ represents the Lichtenbaum \'etale motivic cohomology.

For every integer $i$ and a commutative ring $R$,
let $R(i)$ be the Nisnevich sheaf of complexes on $\Sm/k$ given by the motivic cohomology
\[
X\mapsto R\Gamma_{\mot}(X,R(i)).
\]
Consider again the functor
\[
\omega^*\colon \Sh_{Nis}(\Sm/S,\Sp)\to \Sh_{sNis}(\SmlSm/S,\Sp)
\]
such that $\omega^*\cF(X):=\cF(X-\partial X)$ for $\cF\in \Sh_{Nis}(\Sm/k,\Sp)$ and $X\in \SmlSm/k$.

\begin{prop}
\label{ket.1}
The hypersheaf of complexes $L_{\ket}\omega^* \Z(i)$ on $\SmlSm/k$ is $(\P^n,\P^{n-1})$-invariant for every integer $n>0$ and $i$.
\end{prop}
\begin{proof}
For every complex $\cF$,
we have the arithmetic fracture square
\[
\begin{tikzcd}
\cF\ar[d]\ar[r]&
\prod_\ell \cF_\ell^\wedge\ar[d]
\\
\cF_\Q\ar[r]&
\prod_\ell (\cF_\ell^\wedge)_\Q
\end{tikzcd}
\]
that is cartesian,
where $\ell$ runs over primes,
$(-)_\ell^\wedge$ is the $\ell$-adic completion, and $(-)_\Q$ is the rationalization.
Use this square for
\[
\cF
:=
\fib(L_{\ket}\omega^*\Z(i)(X)\to L_{\ket}\omega^*\Z(i)(X\times (\P^n,\P^{n-1})))
\]
to reduce to showing that the presheaves of complexes $L_{\ket}\omega^*\Q(i)$ and $(L_{\ket}\omega^*\Z(i))_{\ell}^\wedge$ are $(\P^n,\P^{n-1})$-invariant.

If $X_\bullet \to X$ is a Kummer \'etale hypercover,
then $X_\bullet-\partial X_\bullet \to X-\partial X$ is an \'etale hypercover.
By \cite[Theorem 14.30]{MVW},
we have an equivalence $\colim_{i\in \Delta} M(X_i-\partial X_i) \simeq M(X-\partial X)$ in the $\infty$-category of Voevodsky's motives $\mathrm{DM}^\mathrm{eff}(k,\Q)$.
Apply $\map_{\mathrm{DM}^\mathrm{eff}(k,\Q)}(-,\Q(i)[2i])$ to this equivalence to deduce that $X\in \SmlSm/k \mapsto R\Gamma_\mot(X-\partial X,\Q(i))$ is a Kummer \'etale hypersheaf.
It follows that
we have an equivalence of presheaves of complexes with rational coefficients
\(
L_{\ket}\omega^*\Q(i)
\simeq
L_{\sNis}\omega^*\Q(i).
\)
The latter is $(\P^n,\P^{n-1})$-invariant by \cite[Proposition 8.1.12]{BPO}.
Hence it remains to show that $(L_{\ket}\omega^*\Z(i))_{\ell}^\wedge$ is $(\P^n,\P^{n-1})$-invariant.
Since $(L_{\ket}\omega^*\Z(i))_{\ell}^\wedge\simeq \lim_d L_{\ket}\omega^*\Z/\ell^d(i)$,
it suffices to show that $L_{\ket}\omega^*\Z/\ell^d(i)$ is $(\P^n,\P^{n-1})$-invariant.
By induction and using the exact sequence of abelian groups \[0\to \Z/\ell \to \Z/\ell^n \to \Z/\ell^{n-1}\to 0,\]
we reduce to showing that $L_{\ket}\omega^*\Z/\ell(i)$ is $(\P^n,\P^{n-1})$-invariant.

Assume $\ell\neq p$.
We have an equivalence
\[
L_{\ket}\omega^*\Z/\ell(i)
\simeq
L_{\ket}\omega^*\mu_\ell^{\otimes i}
\]
by \cite[Theorem 10.2]{MVW} since the Kummer \'etale topology is finer than the strict \'etale topology.
Hence $L_{\ket}\omega^*\Z/\ell(i)(X)$ is equivalent to the Kummer \'etale cohomology $R\Gamma_{\ket}(X,\Z/\ell(i))$.
This is $\square$-invariant and a log \'etale hypersheaf by \cite[Theorem 9.1.5, Proposition 9.1.6]{BPO2},
so this is $(\P^n,\P^{n-1})$-invariant by \cite[\S 3.5]{BPO2}.

Assume $\ell=p$.
then we have an equivalence of Nisnevich sheaves $\Z/p^n(i)\simeq W_n\Omega_{\log}^i[-i]$ on $\Sm/k$ by Geisser-Levine \cite[Theorem 8.5]{MR1738056}.
Hence we have an equivalence of Kummer \'etale sheaves
\begin{equation}
\label{ket.1.1}
L_{\ket}\omega^*\Z/p^n(i)\simeq L_{\ket}W_n\Omega_{\log}^i[-i]
\end{equation}
on $\SmlSm/k$ since $\omega^*W_n\Omega_{\log}^i\simeq W_n\Omega_{\log}^i$ by the definition of logarithmic forms.
This is $(\P^n,\P^{n-1})$-invariant by Proposition \ref{p.11}.
\end{proof}

\begin{df}
\label{ket.2}
For $X\in \SmlSm/k$,
the \emph{Kummer \'etale motivic cohomology} $R\Gamma_\rL(X,\Z(i))$ is the Kummer \'etale hypersheafification of $X\mapsto R\Gamma_\mot(X,\Z(i)):=R\Gamma_\mot(X-\partial X,\Z(i))$,
i.e.,
\[
R\Gamma_\rL(X,\Z(i))
:=
(L_{\ket}\omega^* \Z(i))(X).
\]
For every integer $i$,
Proposition \ref{ket.1} yields a natural equivalence
\begin{equation}
\label{ket.2.1}
\map_{\logSH_{\ket}^\wedge(k)}(\Sigma^\infty X_+,\Sigma^{0,i}L_\ket \MZ)
\simeq
R\Gamma_{\mathrm{L}}(X,\Z(i)),
\end{equation}
where $\logSH_{\ket}^\wedge(k)$ is the  Kummer \'etale hyper-localization of $\logSH(k)$,
and $L_{\ket}\colon \logSH(k)\to \logSH_{\ket}^\wedge(k)$ is the localization functor. 
The \emph{Kummer \'etale $K$-theory} $L_{\ket}\Kth(X)$ is the Kummer \'etale hypersheafification of $X\mapsto \Kth(X):=\Kth(X-\partial X)$.

Assume that $X$ has a trivial log structure.
Then the small Kummer \'etale site $X_{\ket}$ and small \'etale site $X_{\et}$ agree.
Hence $R\Gamma_{\mathrm{L}}(X,\Z(i))$ agrees with the Lichtenbaum \'etale motivic cohomology, and $L_{\ket}\Kth(X)$ agrees with the \'etale $K$-theory of $X$.

We also consider the localization functor $L_\setale\colon \logSH(k)\to \logSH_\setale^\wedge(k)$.
\end{df}

\begin{df}
For a topology $\tau$ on $\SmlSm/k$,
let $\logSH_\tau(k)_p^\wedge$ be the full subcategory of $\logSH_\tau(k)$ spanned by $p$-complete objects,
and let $(\logSH_\tau(k)_p^\wedge)^\eff$ be the full stable $\infty$-subcategory of $\logSH_\tau(k)_p^\wedge$ generated under colimits by $(\Sigma^{n,0}\Sigma_{\P^1}^\infty X_+)_p^\wedge$ for $X\in \SmlSm/k$ and $n\in \Z$.
We define the slice filtration on $\logSH_\tau(k)_p^\wedge$ as we did on $\logSH_\tau(k)$.
\end{df}

\begin{prop}
\label{ket.4} 
For every integer $i$,
there are equivalences in $\logSH(k)_p^\wedge$
\[
(L_{\setale}\bM \Z(i))_p^\wedge
\simeq
(L_{\ket}\bM \Z(i))_p^\wedge
\simeq
\bM \Z_p^\syn(i).
\]
\end{prop}
\begin{proof}
We have an equivalence in $\logSH(k)$
\[
\bM \Z_p^\syn(i)
\simeq
\lim_m(L_\ket \bW_m \bOmega_{\log}^i[-i])
\]
by Proposition \ref{p.2}.
The right-hand side is equivalent to $(L_\ket\bM \Z(i))_p^\wedge$ by \eqref{ket.1.1}.
To show $(L_{\setale}\bM \Z(i))_p^\wedge\simeq (L_{\ket}\bM \Z(i))_p^\wedge$, use Proposition \ref{p.11} and \eqref{ket.1.1}.
\end{proof}

\begin{thm}
\label{ket.5}
Let $k$ be a perfect field admitting resolution of singularities.
Then the induced morphism
\[
\tilde{\fil}_i^\ket\bTC(-;\Z_p)
\to
\Fil_i^\BMS\bTC(-;\Z_p)
\]
is an equivalence in $\logSH_\ket^\wedge(k)_p^\wedge$ for every integer $i$. Here, we denote by $\tilde{\fil}_i^\ket$ the $i$th very effective cover in $\logSH_\ket^\wedge(k)_p^\wedge$.
We have a similar result for $\tilde{\fil}_i^\setale$ and $\logSH_{\setale}^\wedge(k)_p^\wedge$ too.
\end{thm}
\begin{proof}
We focus on the proof for the Kummer \'etale case since the proof for the strict \'etale case proceeds verbatim.
By Lemma \ref{slice.4} and Theorem \ref{K.4},
we have $L_{\ket}\MZ (i)[2i]\in \logSH_{\ket}^\wedge (k)_{\geq i}^\veff$.
Since the $p$-completion functor $\logSH_{\ket}^\wedge(k)\to \logSH_{\ket}^\wedge(k)_p^\wedge$ is a left adjoint,
we can show that this preserves (very) effectiveness arguing as in Lemma \ref{slice.4}.
Hence we have $\gr_i^\BMS \bTC(-;\Z_p)\in (\logSH_{\ket}^\wedge (k)_p^\wedge)_{\geq i}^\veff$ using Propositions \ref{BMS.9} and \ref{ket.4}.
For $j\geq i$,
we have
\[
\fib(\Fil_j^\BMS \bTC(-;\Z_p)\to \Fil_{i-1}^\BMS \bTC(-;\Z_p))\in (\logSH_{\ket}^\wedge (k)_p^\wedge)_{\geq i}^\veff
\]
by induction on $j$.
Take colimits on $j$ to have $\Fil_i^\BMS \bTC(-;\Z_p)\in (\logSH_{\ket}^\wedge (k)_p^\wedge)_{\geq i}^\veff$.
On the other hand,
Theorem \ref{BMS.3} implies $\Fil^{i-1}_\BMS \bTC(-;\Z_p)\in (\logSH_{\ket}^\wedge (k)_p^\wedge)_{\leq i-1}^\veff$.
It follows that we have $\widetilde{f}^{i-1}\Fil^{i-1}_\BMS \bTC(-;\Z_p)\simeq \Fil^{i-1}_\BMS \bTC(-;\Z_p)$ and hence $\widetilde{f}_i \Fil^{i-1}_\BMS \bTC(-;\Z_p)\simeq 0$.
Now after applying $\widetilde{f}_i$ to the fiber sequence
\[
\Fil_i^{\BMS}\bTC(-;\Z_p)
\to
\bTC(-;\Z_p)
\to
\Fil^{i-1}_{\BMS}\bTC(-;\Z_p),
\]
we obtain an equivalence $\Fil_i^{\BMS}\bTC(-;\Z_p)\simeq \widetilde{f}_i\bTC(-;\Z_p)$.
\end{proof}
\begin{rmk} (1) 
    In the proof of Theorem \ref{ket.5}, we have used in particular the fact that $(L_{\ket}\bM \Z(i))_p^\wedge$ is very effective in the category $\logSH^\wedge_{\ket}(k)^\wedge_p$. Since the inclusion $\logSH^\wedge_{\ket}(k)^\wedge_p \subset \logSH^\wedge_{\ket}(k)$ does not obviously preserve the very effective subcategory, we are not claiming that the same holds when we see $(L_{\ket}\bM \Z(i))_p^\wedge$ as object in $\logSH^\wedge_{\ket}(k)$.  

    (2) One could reasonably ask if the canonical morphism constructed in Theorem \ref{BMS.3} becomes an equivalence for $k$ a field (say, assuming resolution of singularities), after applying the canonical functor from $\logSH(k, \Z_p)$ to $\logSH^\wedge_{\ket}(k)^{\wedge}_p$, and if this equivalence coincides with the one provided by Theorem \ref{ket.5}. Equivalently, one could ask if the $p$-completion of the Kummer étale very effective slice filtration for $\TC(-;\Z_p)$ coincides with the Kummer étale very effective slice filtration computed in $\logSH^\wedge_{\ket}(k)^{\wedge}_p$. We expect this to be the case. 
\end{rmk}
\begin{thm}\label{thm:KummerLich}
Let $k$ be a perfect field admitting resolution of singularities,
and assume that the \'etale cohomological dimension of $k$ is finite.
\begin{enumerate}
\item[\textup{(1)}] The filtrations 
\[\fil_{\bullet}^\ket L_{\ket}\KGL, \quad \tilde{\fil}_{\bullet}^\ket L_{\ket}\KGL, \quad \text{ and } \quad L_{\ket}\tilde{\fil}_{\bullet}\KGL\] on $L_{\ket}\KGL\in \logSH_{\ket}^\wedge(k)$ agree and are complete and exhaustive.
\item[\textup{(2)}]
There are natural equivalences
\[
\slice_i^\ket L_{\ket}\KGL
\simeq
\tilde{\slice}_i^\ket L_{\ket}\KGL
\simeq
L_{\ket}\slice_i\KGL
\simeq
L_{\ket}\tilde{\slice}_i\KGL
\simeq
\Sigma^{2i,i}L_{\ket}\MZ
\]
in $\logSH_{\ket}^\wedge (k)$.
\item[\textup{(3)}]
For $X\in \SmlSm/k$,
there is a natural equivalence
\[
\map_{\logSH_{\ket}^\wedge(k)}(\Sigma^\infty X_+,L_{\ket}\KGL)
\simeq
L_{\ket}\Kth(X).
\]
\item[\textup{(4)}]
For $X\in \SmlSm/k$,
there are natural equivalences
\begin{gather*}
\map_{\logSH_{\ket}^\wedge(k,\Sph/p^n)}(\Sigma^\infty X_+,L_{\ket}\KGL\otimes_{\Sph} \Sph/p^n)
\simeq
L_{\ket}\Kth(X) \otimes_\Sph \Sph/p^n,
\\
\map_{\logSH_{\ket}^\wedge(k)_p^\wedge}(\Sigma^\infty X_+,(L_{\ket}\KGL)_p^\wedge)
\simeq
(L_{\ket}\Kth(X))_p^\wedge
\end{gather*}
without the assumption on the \'etale cohomological dimension.
\end{enumerate}
\end{thm}
\begin{proof}
As pointed out by Spitzweck-{\O}stv{\ae}r, the connectivity of $\tilde{\fil}_n E$ of any motivic spectrum increases with $n$ (see \cite[Proposition 5.11]{SO:twistedK}). Hence the very effective slice filtration is automatically complete.

By Proposition \ref{ket.1},
we have
\[
\map_{\logSH_{\ket}^\wedge(k)}(\Sigma^\infty X_+,\Sigma^{0,i}L_{\ket}\MZ)\simeq 0
\]
for $X\in \SmlSm/k$ and $i<0$.
By \cite[Corollary 5.1.7]{BPO} and \cite[Proposition 2.7.3]{BPO2}, $\Sigma^\infty X_+$ is compact.
Since $L_{\ket}$ preserve colimits,
Proposition \ref{K.4} implies that the filtration $L_{\ket}\tilde{\fil}_\bullet\KGL$ is exhaustive.
It follows that we have
\[
\map_{\logSH_{\ket}^\wedge(k)}(\Sigma^\infty X_+,L_{\ket}\tilde{\fil}^{-1}\KGL)\simeq 0.
\]
This implies
$L_{\ket}\tilde{\fil}^{-1}\KGL\in \logSH_{\ket}^\wedge(k)_{\leq -1}^\eff$.
On the other hand,
Lemma \ref{slice.4} implies $L_{\ket}\tilde{\fil}_0\KGL \in \logSH_{\ket}^\wedge(k)_{\geq 0}^\veff$.
From these,
we have natural equivalences
\[
\fil_0 L_{\ket}\KGL
\simeq
\tilde{\fil}_0 L_{\ket}\KGL
\simeq
L_{\ket}\tilde{\fil}_0\KGL.
\]
Apply $\Sigma^{2i,i}$ to this and use Proposition \ref{K.4} to finish the proof of (1).

(2) is an immediate consequence of (1).

Consider the localization functor
\[
L_{\ket}
\colon
\Sp_{\P^1}(\Sh_{sNis}(\SmlSm/k,\Sp))
\to
\Sp_{\P^1}(\Sh_{\ket}(\SmlSm/k,\Sp))
\]
without the further $(\P^n,\P^{n-1})$-localizations.
To show (3),
we only need to show that $L_{\ket}\KGL$ is $(\P^n,\P^{n-1})$-invariant for every integer $n$.
The above argument shows that the filtration $L_{\ket}\tilde{f}_{\geq \bullet}\KGL$ on $L_{\ket}\KGL$ is complete and exhaustive since this argument does not require $(\P^n,\P^{n-1})$-localizations.
Hence it suffices to show that the graded pieces $L_{\ket}\tilde{\slice}_i \KGL$ are $(\P^n,\P^{n-1})$-invariant.
This is a consequence of Theorem \ref{K.4} and Proposition \ref{ket.1}.

Argue similarly as above and use $\lim_n$ for (4),
but note that the \'etale $\Z/p^n$-cohomological dimension (=\'etale $\Sph/p^n$-cohomological dimension) of $k$ is always finite by \cite[Th\'eor\`eme X.5.1]{SGA4}.
\end{proof}

\begin{rmk}
Even though $\KGL\in \logSH(k)$ is $\A^1$-invariant,
$L_{\ket}\KGL\in \logSH_{\ket}^\wedge(k)$ is not $\A^1$-invariant in the above case of $k$ with $\chara k>0$ since $R\Gamma_\rL$ is not $\A^1$-invariant as observed in \cite[Lecture 10]{MVW}.
\end{rmk}

\begin{const}
\label{ket.3}
Let $k$ be a perfect field admitting resolution of singularities.
By Proposition \ref{p.11} and \eqref{p.6.1},
log syntomic cohomology on $\SmlSm/k$ is a Kummer \'etale hypersheaf.
It follows that the log motivic cyclotomic trace map 
\[
\Tr\colon \KGL\to \bTC
\]
in $\logSH(k)$ yields the map
\[
L_{\ket}\Tr \colon L_{\ket}\KGL\to \bTC
\]
in $\logSH_{\ket}^\wedge(k)$.

If we apply $\map_{\logSH_\ket^\wedge(k)}(\Sigma^\infty X_+,-)$ to $L_{\ket}\Tr$,
then we get
\[
L_{\ket}\Kth(X)
\to
\TC(X)
\]
by Theorem \ref{thm:KummerLich}(3) when the \'etale cohomological dimension of $k$ is finite,
which is the Kummer \'etale hypersheafification of the log cyclotomic trace map $\Kth(X)\to \TC(X)$.
Similarly,
if we apply $\map_{\logSH_\ket^\wedge(k)_p^\wedge}(\Sigma^\infty X_+,-)$ to $(L_{\ket}\Tr)_p^\wedge$,
then we get
\[
(L_{\ket}\Kth(X))_p^\wedge
\to
\TC(X;\Z_p)
\]
by Theorem \ref{thm:KummerLich}(4).
\end{const}

As an application of Kummer \'etale $K$-theory,
we show the following:

\begin{prop}\label{prop:GeissHessEt}
Let $k$ be a perfect field admitting resolution of singularities.
For $X\in \SmlSm/k$,
there is a natural equivalence
\[
(L_{\ket}\Kth(X))^{\wedge}_p
\simeq
\TC(X;\Z_p).
\]
\end{prop}
\begin{proof}
Assume first that $X\in \Sm/k$. Then the equivalence $(L_{\et}\Kth(X))^{\wedge}_p \simeq \TC(X;\Z_p)$ is due to Geisser and Hesselholt \cite[Theorem 4.2.6]{GeissHess1}.
In general, the desired equivalence for $X\in \SmlSm/k$ follows from the Gysin sequence \cite[Theorem  3.2.21]{BPO2} by induction on the number of smooth irreducible components $r$ of $\partial X$. More precisely, if $Z$ is a smooth divisor on proper $X\in \Sm/k$, then we have the fiber sequence in $\logSH(k)$
\begin{equation}\label{eq:loc}\Sigma^\infty ( X , Z)_+  \to \Sigma^\infty X_+ \to \mathrm{Th}(\rN_Z X), \end{equation}
where $\mathrm{Th}(\rN_Z X)$ is the motivic Thom space defined in \cite[Definition 3.2.7]{BPO2}.  Note that by applying the natural functor $\omega_\sharp \colon \logSH(k)\to \SH(k)$, left adjoint to $\omega^*$, the above sequence reduces to the standard localization sequence in $\SH(k)$. In particular, applying the log $K$-theory spectrum $\KGL$, we obtain the localization sequence in algebraic $K$-theory in light of \cite[Theorem 6.5.7]{BPO2}. 
On the other hand, if we apply the log TC spectrum $\mathbf{TC}$ to \eqref{eq:loc} we obtain by definition the Gysin or residue sequence in logarithmic topological cyclic homology.
By Construction \ref{ket.3},
we obtain a commutative diagram of spectra
\[
\begin{tikzcd}
  (L_{\ket}\Kth(Z))^\wedge_p \arrow[r]\arrow[d] & (L_{\ket}\Kth(X))^\wedge_p  \arrow[r]\arrow[d]   & (L_{\ket}\Kth(X,Z))^\wedge_p \arrow[d]\\
\TC(Z;\Z_p) \arrow[r]& \TC(X;\Z_p) \arrow[r] & \TC((X,Z);\Z_p)
\end{tikzcd}
\]
whose horizontal sequences are fiber sequences.
The general case follows similarly by induction on $r$.
\end{proof}

\begin{rmk}\label{rmk:Licht_mot}
Assume that the boundary $\partial X$ of $X\in \SmlSm/k$ has $r$ irreducible components.
For integers $m\geq 1$ and $n\leq d-r$,
the induced natural map
\[
\pi_n(R\Gamma_\mot(X,\Z/p^m(d)))
\to
\pi_n(R\Gamma_\rL(X,\Z/p^m(d)))
\]
is an isomorphism.
Indeed, if $r=0$,
then this is a consequence of \cite[Theorem 8.5]{MR1738056}.
If $r>0$,
then proceed by induction on $r$,
and use the Gysin sequence \cite[Theorem  3.2.21]{BPO2} and the five lemma.
See \cite[Corollary 1.1]{zbMATH07020394} for a related result with $\Q_p$-coefficients and the inequality $n\leq d$.
\end{rmk}

\begin{rmk}
   The fact that $p$-adic \'etale $K$-theory is identified with topological cyclic homology holds in larger generality (namely, without smoothness assumption) in light of \cite[Theorem C]{CMM}. However, this does not immediately imply a generalization of Proposition \ref{prop:GeissHessEt}. 
\end{rmk}

\bibliography{references}
\bibliographystyle{siam}

\end{document}